\documentclass[10pt,a4paper]{article}
\usepackage[utf8]{inputenc}
\usepackage[english]{babel}
\usepackage{amsmath}
\usepackage{amsfonts}
\usepackage{amssymb}
\usepackage{graphicx}
\usepackage{amsthm}
\usepackage{lipsum}
\usepackage{csquotes}
\usepackage{blindtext} 
\usepackage{mathtools}
\usepackage{thmtools}
\usepackage[numbers]{natbib}

\usepackage{braket}
\usepackage{cancel}
\usepackage{xcolor}

\usepackage{floatrow}
\usepackage{setspace}
\usepackage[margin=1in]{geometry}
\allowdisplaybreaks[1]

\mathtoolsset{showonlyrefs}

\newtheorem{theorem}{Theorem}[section]
\newtheorem{corollary}{Corollary}[section]
\newtheorem{lemma}{Lemma}[section]
\newtheorem{definition}{Definition}
\newtheorem*{remark}{Remark}

\newcommand{\bit}{\begin{itemize}}
\newcommand{\eit}{\end{itemize}}
\newcommand{\bd}{\begin{description}}
\newcommand{\ed}{\end{description}}
\newcommand{\ben}{\begin{enumerate}}
\newcommand{\een}{\end{enumerate}}
\newcommand{\bt}{\begin{thm}}
\newcommand{\et}{\end{thm}}
\newcommand{\bl}{\begin{lem}}
\newcommand{\el}{\end{lem}}
\newcommand{\bp}{\begin{prop}}
\newcommand{\ep}{\end{prop}}
\newcommand{\bc}{\begin{cor}}
\newcommand{\ec}{\end{cor}}
\newcommand{\bdefn}{\begin{defn}}
\newcommand{\edefn}{\end{defn}}
\newcommand{\brem}{\begin{rem}}
\newcommand{\erem}{\end{rem}}
\newcommand{\bproof}{\begin{proof}}
\newcommand{\eproof}{\end{proof}}
\newcommand{\bex}{\begin{ex}}
\newcommand{\eex}{\end{ex}}
\newcommand{\R}{\mathbb{R}}

\newcommand{\bcs}{\begin{cases}}
\newcommand{\ecs}{\end{cases}}

\newcommand{\spc}{\;\;\;}

\newcommand{\expp}{\text{e}}

\newcommand{\se}{\text{SE}(2)}

\newcommand{\SUM}{\sum\limits}

\newcommand{\eps}{\epsilon}
\newcommand{\om}{\omega}

\newcommand{\hxi}{\hat{\xi}}
\newcommand{\het}{\hat{\eta}}
\newcommand{\htt}{\hat{t}}

\newcommand{\abs}[1]{\lvert #1\rvert}
\newcommand{\abss}[1]{\big\lvert #1 \big\rvert}

\newcommand{\jetp}{\tilde{\mathcal{P}}^{2,+}}
\newcommand{\jetn}{\tilde{\mathcal{P}}^{2,-}}

\date{}
\author{E. Baspinar\thanks{University of Bologna, Department of Mathematics.} \and G. Citti\thanks{University of Bologna, Department of Mathematics.}\footnotemark[1]}
\title{
Uniqueness of viscosity mean curvature flow solution \\ in two
 sub-Riemannian structures}
\begin{document}
\maketitle
%

\section{Introduction}

Mean curvature flow describes the evolution of a surface whose points move in the normal 
direction, with speed equal to the curvature. The first results in the Euclidean setting have been provided by Gage \cite{gage1983isoperimetric, gage1984curve}, Huisken \cite{huisken1984flow}, Gage-Hamilton \cite{gage1986heat}, Grayson \cite{grayson1987heat} and Altschuler-Grayson \cite{altschuler1991shortening}, with the methods based on differential geometry. 
Since a mean curvature flow can develop singularities even for initially smooth surfaces, (see for example \cite{evans1991motion}), different notions of weak solution were proposed in order to study the flow after singularities: Brakke introduced in \cite{brakke1978motion} an approach based on the notion of varifold and geometric measure theory, Evans-Spruck \cite{evans1991motion}, \cite{evans1992motion2}, \cite{evans1992motion3}, \cite{evans1995motion} and Chen-Giga-Goto \cite{chen1991uniqueness} independently studied existence and uniqueness of viscosity solutions via level set methods. 

The level set method identifies the evolving surface at time $t$ as a level sext $M_t=\{x\in\R^n:\;u(x,t)=0  \}$ of a function $u$, which is a solution to a differential equation. In $n-$dimensional Euclidean setting the curvature can be expressed as 
 $\displaystyle K=\operatorname{div}\Big(\frac{\nabla u}{\abs{\nabla u}}\Big)$
and the mean curvature flow equation reads:
\begin{align}\label{eq:euclideanMCF}
\partial_t u(x,t)=\abs{\nabla u}K=\SUM_{i,j=1}^n\Big(\delta_{ij}-\frac{\partial_{x_i}u\,\partial_{x_j}u}{\abs{\nabla u}^2} \Big)\partial_{x_i x_j}u,
\end{align}
where $\delta_{ij}$ denotes the Kronecker delta operator, $\partial_t$ represents the derivative with respect to time variable and $\partial_{x_i}$ the derivative with respect to $i^{\text{th}}$ spatial variable. 
 
The present paper focuses on the sub-Riemannian 
analogue of the mean curvature flow in a group $G$, which can be either a Carnot group of step 2 or the group $\se$ of rigid motions. A sub-Riemannian structure on one of these groups is defined by a triple $(G,\mathcal{D},(g_{ij}))$, where $G$ is 
the group, $\mathcal{D}$ is a horizontal distribution and $(g_{ij})$ is a metric on $\mathcal{D}$. The space has the bracket generating properties at step $2$. That is, if we denote a basis of $\mathcal{D}$ by $\{X_1, X_2,\dots, X_m\}$ then $\{[X_i, X_j]\}_{i,j=1,2\dots,m}$ together with $\{X_1, X_2, \dots, X_m \}$ spans the tangent space $TG$ to $G$ at every point where $[.\,,\,.]$ is the Lie bracket. Eventually we will choose a metric $(g_{ij})$ on $\mathcal{D}$, which will make $X_1,X_2,\dots, X_m$ orthonormal. 
In the particular case where we consider $\se$, the underlying manifold will be 
$G = \R^2 \times S^1$, its elements will be expressed by
$\xi=(x, y, \theta)\in\se$ so that $x,\; y$ denote the spatial components and $\theta$ the angular (orientation) component. We will make the choice of the vector fields 
\begin{align}\label{eq:seHVfs}
X_1=\cos(\theta)\partial_x+\sin(\theta)\partial_y,\quad X_2=\partial_{\theta},
\end{align} 
at every $\xi\in G$, which satisfy the bracket generating condition, as it is easy to verify. We will denote the commutator by
\begin{equation}\label{eq:seVVfs}
X_{3}=-\sin(\theta)\partial_x+\cos(\theta)\partial_y.
\end{equation}
While studying a Carnot group $G\simeq \R^n$ of step 2 we will denote the corresponding elements by $\xi= (x,\theta)$, with $x=(x_1,x_2,\dots,x_m)$ representing the horizontal variables and $\theta=(\theta_1,\theta_2,\dots,\theta_{n-m})$ representing the variables of the second layer. It is known (see for example \cite{rothschild1976hypoelliptic}) that a basis of 
bracket generating vector fields can be represented by using $m\times m$ real matrices $W^{(i)}$ in this setting as:
\begin{align}\label{eq:VFs_representation_carnot}
\begin{split}
X_i=\partial_{x_i}+\langle(W x)_i, \nabla_\theta \rangle=\partial_{x_i}+\SUM_{k=1}^{n-m}\SUM_{l=1}^m w_{il}^{(k)} x_l\partial_{\theta_k},\quad i=1,\dots,m,
\end{split}
\end{align}
with $\nabla_\theta=(\partial_{\theta_1},\dots,\partial_{\theta_m})$, $(W x)_i=((W^{(1)}x)_i,\dots,(W^{(m)}x)_i)$ where $(W^{(k)}x)_i$ is the $i^{\text{th}}$ component of $W^{(k)}x$  (see also \cite{bonfiglioli2007stratified}, \cite{capogna2009generalized} and \cite{bonfiglioli2005note}). Note that up to a Lie group isomorphism the matrices $W^{(k)}$ can be assumed to be skew-symmetric (see for example \cite{bonfiglioli2005note} and \cite{arena2010taylor} for more details). Therefore we will consider that
\begin{equation}\label{eq:skewsymmetric_W}
w_{ij}^{(k)}=-w_{ji}^{(k)},\quad w_{ii}^{(k)}=0.
\end{equation}
Furthermore the vertical vector fields in the second layer can be expressed as
\begin{equation}\label{eq:vertical_vfs}
X_{i}=\partial_{\theta_{i-m}},\quad i=m+1,\dots,n.
\end{equation}

We explicitly note that that the Heisenberg group can be considered as the limit structure obtained 
from $\se$ via a blow up procedure (see \cite{chen1991uniqueness}, \cite{evans1993convergence} and \cite{franchi2003regular}), hence those two structures, 
which have completely different group laws, share the same local structure. 
This is why they can be studied together, and used as models of the same types of problems. The interest in studying motion by curvature in these two groups comes from applications of image inpainting through models of the visual cortex. Recall that the first layer (V1) of the mammal visual cortex 
was modeled as a smooth sub-Riemannian surface with the local structure of the Heisenberg group in
\cite{petitot1999vers}, and with the $\se$ geometry in \cite{citti2006cortical}. As a consequence, some models of image 
completion inspired by the functionality of the cortex were proposed in \cite{merriman1992diffusion} and \cite{citti2006cortical}. Convincing completion results have been presented using the mean curvature flow 
in these groups (see \cite{sanguinetti2008implementation} and \cite{citti2016sub}). 
We also remark that numerous image processing applications can be performed in similar
sub-Riemannian geometries (see for example the algorithms proposed in \cite{franken2009crossing}, \cite{duits2010left}, \cite{duits2010left2}, \cite{duits2011left}, \cite{duits2013evolution}, \cite{bekkers2014multi}). 

The main obstacle to the development of a strong numerical theory 
for sub-Riemannian mean curvature flow is the lack of uniqueness results in those settings. Indeed the existence of sub-Riemannian mean curvature flow solutions is known in Carnot groups (see \cite{capogna2009generalized}), and in general H\"ormander structures \cite{dirr2010evolution}, but the uniqueness problem is still largely open. Furthermore, the geometry in Carnot groups of step 2 is different from the geometry in Carnot groups of a higher order step, since it was shown in \cite{capogna2009generalized} that the planes are not minimal surfaces in Carnot groups of a step strictly bigger than two, resulting in the lack of a family of functions which can be used as barrier functions. Hence it is natural to focus only on step 2 groups. Up to now  Capogna and Citti \cite{capogna2009generalized} proved the uniqueness of evolving graphs in a Carnot group by using the fact that graphs do not suffer from singularities during the mean curvature flow. In the special case of the 3-dimensional Heisenberg group, Ferrari, Liu and Manfredi \cite{ferrari2014horizontal} provided uniqueness under the assumption of axisymmetricity of solutions to the sub-Riemannian mean curvature flow equation given by \eqref{eq:orgMCF} in the sequel. 

Here we will present a complete proof of the uniqueness of the sub-Riemannian mean curvature flow solution in the generic setting of Carnot group of step 2 and in the setting of $\se$ by discarding the previous restrictions on the solution.

Differential calculus in sub-Riemannian spaces is well-established 
(see for example \cite{montgomery2006tour}). The horizontal gradient of a function is defined as
\begin{align}
\nabla_0 =(X_1,X_2,\dots, X_m). 
\end{align}
The notion of regular surface in sub-Riemannian settings has been introduced by 
Franchi, Serapioni and Serra Cassano \cite{franchi2003regular} as the zero level set $M=\{\xi \in G: u(\xi)=0 \}$ of a smooth function $u$, whose horizontal gradient does not vanish. 
However even surfaces which are regular in the Euclidean sense
have points at which the 
horizontal gradient vanishes. We call such points \emph{characteristic points} and we denote the set of those points by $\Sigma(M)=\{\xi\in M:\; \abss{\nabla_0 u(\xi,t)}=0\}$. As a result of the presence of characteristic points we lack a definition of mean curvature in sub-Riemannian settings. 

Let us recall that at non-characteristic points the horizontal normal is defined as 
$$\nu_0=\frac{\nabla_0 u}{\abs{\nabla_0 u}},$$
and the horizontal mean curvature of the manifold $M$ is given by
\begin{align}
K_0=\SUM_{i=1}^m X_i\nu_{0,i}.
\end{align}
This notion has been introduced in a general setting in \cite{danielli2007sub}. Results on the mean curvature equation in the special setting of the Heisenberg group were provided in \cite{capogna2007introduction}, \cite{capogna2009generalized}, \cite{capogna2013sub}, \cite{ferrari2014horizontal}
and in $\se$ by \cite{citti2016sub}. We refer to \cite{capogna2009generalized} for more detailed references.

Analogously to the Euclidean setting, the horizontal mean curvature flow is the evolution of a surface $M_0\subset G$ with normal speed equal to the horizontal mean curvature. If $M_0$ is the level set of a function $u_0$, the flow at time $t$ will be identified as the level set $M_t=\{\xi\in G:\; u(\xi,t)=0\}$ of the solution of the following degenerate problem:
\begin{align}\label{eq:orgMCF}
\bcs
\partial_t u=\SUM_{i,j=1}^m\Big(\delta_{ij} - \frac{X_i u X_j u}{|\nabla_0 u|^2}\Big)(X_iX_j)u\quad\text{in $G\times (0,\infty)$}\\
u(.,0)=u_0(.)\quad \text{on $G\times \{ 0 \}$}.
\ecs
\end{align}

In order to establish the uniqueness we prove the following comparison 
principle for the viscosity solutions of problem \eqref{eq:orgMCF} (see Definition \ref{defn:viscosityDefn} below):
\begin{corollary}\label{thm:comparisonPrincipleViscosity}
Assume that $u$ and $v$ are continuous viscosity solutions of 
problem \eqref{eq:orgMCF} (as defined in Definition \ref{defn:viscosityDefn}) such that 
\begin{itemize}
\item{there exists $R>0$ with $u=0$, $v=0$ for $|\xi|>R$},
\item{$u\leq v$ at time $t=0$.}
\end{itemize}
Then $u\leq v$ for every $t>0$. 
\end{corollary}
Note that the role of the boundary conditions on the solutions is to restrict the solution level set to stay in a bounded set. In fact any constant $C$ could be chosen as the boundary condition without losing the generality. We choose zero as boundary condition for the sake of simplicity when we study the asymptotic behavior of the solutions.

The proof relies on a regularization procedure together with a Riemannian approximation. We will use also the general notation which includes the approximating vertical vector fields given by
\begin{equation}\label{eq:generalVFs}
X_{i\delta}=\delta^{\operatorname{deg}i-1}X_i,\quad i=1,\dots,n,
\end{equation} 
where  $\delta>0$ and $\operatorname{deg(\cdot)}$ gives the degree of its argument, i.e.,
\begin{equation}
\operatorname{deg}i=\begin{cases}
1\quad \text{if}\quad i\leq m,\\
2\quad \text{if}\quad m+1\leq i\leq n.
\end{cases}
\end{equation}
We define the gradient
\begin{equation}\label{eq:defin_approximating_gradient}
\nabla_{\delta}=(X_1, X_2,\dots,X_m, X_{(m+1)\delta},\dots,  X_{n\delta}),
\end{equation}
with $0<\delta,\eps<1$, then for an open ball $B(0, R)$ (with respect to the left-invariant metric generated by $X_1,\dots,X_m$) with radius $R$ centered at the origin we introduce the regularized problem with the Riemannian approximation 
\begin{align}\label{eq:regMCFbounded}
\bcs\partial_t u^{\eps}_{\delta}=\SUM_{i,j=1}^n\Big(\delta_{ij} - \displaystyle\frac{X_{i\delta} u^{\eps}_{\delta} X_{j\delta}u^{\eps}_{\delta}}{|\nabla_\delta u^{\eps}_{\delta}|^2 + \epsilon^2}\Big)X_{i\delta}X_{j\delta}u^{\eps}_{\delta}\quad\text{in $B(0, R)\times (0,\infty)$}\\
u^{\eps}_{\delta}(.,0)=u_0(.)\quad \text{on $B(0, R)\times \{ 0 \}$}\\
u^{\eps}_{\delta}(.,t)=0\quad \text{on $\partial B(0, R)\times  [0, T] $},
\ecs
\end{align}
and the same problem on the whole space
\begin{align}\label{eq:regMCF}
\bcs
\partial_t u^{\eps}_{\delta}=
\SUM_{i,j=1}^n\Big(\delta_{ij} - \displaystyle\frac{X_{i\delta} u^{\eps}_{\delta} X_{j\delta}u^{\eps}_{\delta}}{|\nabla_\delta u^{\eps}_{\delta}|^2 + \epsilon^2}\Big)X_{i\delta}X_{j\delta}u^{\eps}_{\delta}\quad\text{in $G\times (0,\infty)$}\\
u^{\eps}_{\delta}(.,\,0)=u_0(.)\quad \text{on $G\times \{ 0 \}$}.
\ecs
\end{align}

\begin{remark}
Here we explicitly remark that the solutions $u^{\eps}_{\delta}$ to the equations given by \eqref{eq:regMCFbounded} and \eqref{eq:regMCF} are of $C^{\infty}$ type and they are considered as the solutions in the classical sense.
\end{remark}

Then we use the following \emph{vanishing viscosity} solution definition: 
\begin{definition}\label{vvisc}
A function $u$ is called a bounded, continuous vanishing viscosity solution to \eqref{eq:orgMCF} if there exist a constant $\mathcal{C}$, for any compact set $Z$ a constant $C(Z)$, and for any $\delta,\,\eps,\,R>0$ a solution $u^{\eps, R}_{\delta}$ to the problem \eqref{eq:regMCFbounded} such that 
\begin{itemize}
\item{
$||u^{\eps, R}_{\delta}||_{\infty}  \leq \mathcal{C}, \quad |\nabla_E u^{\eps, R}_{\delta}(\xi, t)|\leq C(Z)$ where $\nabla_E$ denotes the Euclidean gradient,}
\item{there exists a function  $u^{\eps}_{\delta}$, solution to \eqref{eq:regMCF}, such that 
$u^\eps_\delta =\displaystyle\lim_{R\rightarrow +\infty}u^{\eps, R}_\delta
$ uniformly,}
\item{
$u=\displaystyle \lim_{\eps,\delta\rightarrow 0}u^{\eps}_{\delta}$ uniformly.}
\end{itemize} 
\end{definition}

This definition immediately implies that comparison principle and uniqueness are 
valid for vanishing viscosity solutions (see also Theorem \ref{thm:comparisonForVanishingViscosity} below and \cite{citti2016sub} for more details). The same result will 
be extended to viscosity solutions, proving that 
those two notions, viscosity and vanishing viscosity solutions, 
coincide. This assertion has been already known in the Euclidean setting but not in a sub-Riemannian setting. 
Indeed the crucial idea of the article is to establish the following 
approximation result:
\begin{theorem}\label{thm:univiscosityTheorem}
Let $v$ be a bounded, continuous viscosity solution to the problem \eqref{eq:orgMCF} in the sense of Definition \ref{defn:viscosityDefn} below, constantly equal to $0$ outside of a compact set, and let $u^\epsilon_{\delta}$ be a solution to the problem \eqref{eq:regMCF}, limit of problem 
\eqref{eq:regMCFbounded}. Then for every $0<T<\infty$ and $\alpha>0$ there exists a constant $M=M(u_0,T,\alpha)$ such that
\begin{align}
\sup_{\xi\in G, 0\leq t\leq T}\lvert (v-u^{\eps}_{\delta})(\xi, t)\rvert\leq M\eps^{\alpha},
\end{align}
for all $0<\eps<1$ and $\delta=\delta(\eps)$. 
\end{theorem}

An immediate consequence of Theorem \ref{thm:univiscosityTheorem} is the following:
\begin{corollary}Any continuous 
viscosity solution $v$ is a vanishing viscosity solution. 
\end{corollary}

The proof of Theorem \ref{thm:univiscosityTheorem} generalizes the proof of Deckelnick \cite{deckelnick2000error} to the sub-Riemannian geometries of step 2 Carnot groups and $\se$. The comparison principle of viscosity
solutions is based on the maximization of 
\begin{equation}\label{wvmenophi}
w(\xi, \eta, t) = v(\xi,t) - u(\eta, t) - \phi(\xi,\eta, t),
\end{equation}

for a suitable 
function $\phi$ where $\xi,\eta\in G$. In Euclidean setting (see for example \cite{deckelnick2000error}) the function $\phi$ satisfies the following conditions:

i) the sum of the second order derivatives vanish 
when the first order ones do, 

ii) $X_i^\xi\phi=-X_i^\eta\phi$, where $X_i^{\xi}$ and $X_i^{\eta}$ denote the $i^\text{th}$ horizontal vector fields at $\xi,\eta\in G$.

As noted by Ferrari, Liu and Manfredi \cite{ferrari2014horizontal}, the main difficulty in extending the classical approach to the sub-Riemannian setting is to find
a function $\phi$ satisfying the above mentioned conditions.
We handle this difficulty and we consider a weaker condition required on  $\lvert X_i^\xi\phi+X_i^\eta\phi\rvert$ (see \eqref{firstderivphi}). Then instead of following the classical approach for proving uniqueness, which is based on comparing two different solutions of the same equation, we compare two different solutions of two different equations, \eqref{eq:orgMCF} and \eqref{eq:regMCF}, where the latter converges to the first one at the limit as we send the parameters $\eps$ and $\delta$ to zero.    

Note that with the approximated equation \eqref{eq:regMCF} we give a formal meaning 
to the regularized operator $|\nabla_\delta u^\eps_\delta| K_0$ at the characteristic points. Indeed, formally if $\xi$ is characteristic 
for every $\epsilon >0$, then
\begin{equation}\label{eq:characteristic_gives_Laplace}
\SUM_{i,j=1}^n\Big(\delta_{ij} - \frac{X_{i\delta} u^\eps_\delta(\xi) X_{j\delta} u^\eps_\delta(\xi)}{|\nabla_\delta u^\eps_\delta(\xi)|^2 + \epsilon^2}\Big)X_{i\delta}X_{j\delta}(u^\eps_\delta(\xi)) =\SUM_{i=1}^m X_iX_iu(\xi),
\end{equation}
which is the Laplace operator of the function $u$ in the sub-Riemannian setting of $G$ as $\delta,\eps\rightarrow 0$. 

The article is organized as follows. In Section \ref{sec:DefnsPreliminaries}, we provide some basics of $\se$ and Carnot group sub-Riemannian geometries. In particular, we provide the notions of vanishing viscosity solution and viscosity solution, and describe the generalized flow. In Section \ref{sec:asymptoticProp} we give some geometric properties of vanishing viscosity solutions. Finally in Section \ref{sec:viscosityUniqueness} we prove Theorem \ref{thm:univiscosityTheorem} in both settings of Carnot groups and $\se$. From this theorem we deduce the comparison principle and the uniqueness result for viscosity solutions in those settings.

\section{Definitions and preliminary results}\label{sec:DefnsPreliminaries}

\subsection{Definition of viscosity and vanishing viscosity solutions}\label{sec:superjetDefns}

Let us consider a vector field $X$, a point $\xi$ in $G$. 
If $\gamma_{\xi, X}$ is a solution to the following Cauchy problem:
\begin{equation}
\begin{cases}
\dot{\gamma}_{\xi, X}(t)= X \gamma_{\xi, X}(t)\\
\gamma_{\xi, X}(0)=\xi,
\end{cases}
\end{equation}
then we will define  $\exp(X)(\xi):=\gamma_{\xi, X}(1)$. 
For every fixed $\xi$ and $e=(e_1, e_2,\dots, e_n)\in \R^n$ the exponential map 
\begin{equation}\label{eq:exponential_map_displacements}
e \mapsto  \exp(\sum_{i=1}^n e_i X_i)(\xi),
\end{equation}
is a local diffeomorphism from a neighborhood of $0$ in $\R^n$ to a neighborhood of 
$\xi$. 

In Carnot groups, using the representation of the vector fields provided in  \eqref{eq:generalVFs},
a direct computation shows that the increment defined in 
\eqref{eq:exponential_map_displacements} can be expressed as
\begin{align}\label{eq:carnot_increment_exp}
\begin{split}
e_i= & x_{\xi,i}-x_{\eta,i}\quad \text{if}\quad \operatorname{deg}i=1,\\
e_i= & \theta_{\xi,i-m}-\theta_{\eta,i-m}+\frac{1}{2}\SUM_{l,j=1}^m w_{lj}^{(i-m)}(x_{\xi,j}x_{\eta,l}-x_{\xi,l}x_{\eta,j})\quad \text{if}\quad \operatorname{deg}i=2.
\end{split}
\end{align}

On the other side, in the setting of $\se$, in order to simplify the computations, 
we will consider the increments associated to the vector fields with constant coefficients.  
Starting from \cite{ambrosio2006intrinsic}, it is indeed clear how the exponential distance \eqref{eq:exponential_map_displacements} induced by a family of vector fields is  
approximated with the distance induced by vector fields with constant coefficients. 
We consider the distance induced by the vector fields with coefficients evaluated 
at the point $\xi$, then we symmetrize the distance by following a similar idea as the one proposed in \cite{citti2006implicit}, \cite{citti2013poincare}. Here we choose $\theta_0=\theta_\xi$. Then the exponential increments becomes: 
\begin{align}\label{eq:se_increment_exp}
\begin{split}
e_1=\cos(\theta_0)&(x_{\xi}-x_{\eta})+ \sin(\theta_0)(y_{\xi}-y_{\eta}), \quad e_2=\sin(\theta_{\xi}-\theta_{\eta}),\\
 & e_3=-\sin(\theta_0)(x_{\xi}-x_{\eta})+\cos(\theta_0)(y_{\xi}-y_{\eta}).
\end{split}
\end{align}

We define the superjets $\mathcal{P}^{2,+}u(\xi,t)$ and $\mathcal{P}^{2,-}u(\xi,t)$, 
as follows:
\begin{align}\label{eq:defnViscositySub}
\begin{split}
\mathcal{P}^{2,+}u(\xi,t):=\{(a,p,H)\in \R\times \R^n \times \mathcal{S}(m)\; \vert \; u\Big(\exp\big(\SUM_{i}^n e_i X_i \big),s\Big)(\xi)\leq u(\xi,t)\\
+a(s-t)+ \SUM_{i}^n  p_i e_i + \frac{1}{2}\SUM_{i}^m H_{ij}e_ie_j  + o(\abs{s-t}+\abs{e}^2)\quad \text{as}\quad(e ,s-t)\rightarrow 0   \},
\end{split}
\end{align}
and $\mathcal{P}^{2,-}u(\xi,t)=-\mathcal{P}^{2,+}(-u)(\xi,t)$ with $\mathcal{S}(n)$ representing the group of $n\times n$ symmetric matrices. 
Note that 
$(p_1,p_2,\dots,p_m)$ plays the role analogous to a horizontal gradient in this formula. However in a Carnot group of step 2, the right hand side contains the analogue of the complete gradient 
$(p_1,p_2,\dots,p_n)$. Furthermore $(H_{ij})_{i,j=1,2,\dots, m}$ plays the role of a horizontal Hessian. Finally we denote the closure of the superjets by $\tilde{\mathcal{P}}^{2,+}$ and $\tilde{\mathcal{P}}^{2,-}$. 

Then the viscosity solution in this case is defined as follows:
\begin{definition}\label{defn:viscosityDefn}
A function $u\in C^0 (G\times [0,\infty))$ is called a viscosity subsolution of \eqref{eq:orgMCF} if for every $(\xi_0,t_0)\in G\times (0,\infty)$ and every $(a,p,H)\in \jetp u(\xi_0,t_0)$ it is provided that
\begin{align}
a &\leq \SUM_{i,j=1}^m\Big (\delta_{ij}-\frac{p_ip_j}{\abs{p}^2} \Big)H_{ij}\quad \text{if}\quad p\neq 0,\\
a &\leq \SUM_{i,j=1}^m \delta_{ij}H_{ij}\quad\text{if}\quad p=0.
\end{align}
Viscosity supersolution is defined analogously where  $\leq$ is replaced by $\geq$ and $\jetp u(\xi_0,t_0)$ by $\jetn u(\xi_0,t_0)$. A viscosity solution to \eqref{eq:orgMCF} is a function $u\in C^0(G\times [0,\infty))$ which is both a subsolution and supersolution.
\end{definition}
The condition at the characteristic points is motivated by \eqref{eq:characteristic_gives_Laplace}. Let us explicitly note that alternative definitions are available, as for example given in \cite{dirr2010evolution}.

We explicitly remark that the set of viscosity solutions is a larger set than the set of vanishing viscosity solutions. In other words, a vanishing viscosity solution to \eqref{eq:orgMCF} is the viscosity solution which is the limit of a solution to \eqref{eq:regMCF} as $\epsilon,\delta\rightarrow 0$. Consequently vanishing viscosity solutions are also viscosity solutions (see \cite{citti2016sub}) while the fact that viscosity solutions are also vanishing viscosity solutions will be proved here in this article.

\subsection{Existence and comparison results}

In \cite{evans1991motion} the existence of a vanishing viscosity solution of an
Euclidean mean curvature flow was established under the assumption that the initial condition is identically $1$ at infinity. The same theorem is already known in the two types of groups considered here: Carnot groups and $\se$ (see \cite{capogna2009generalized} for Carnot groups and \cite{citti2016sub} for $\se$).

\begin{theorem}\label{existence}
Assume that $G$ is either a Carnot group or $\se$, and the initial datum $u_0$ is of class $C^1_E(G)$ (i.e., in the Euclidean sense). Then there is a sphere of radius $R$ such that $u_0$ is identically constant out of this sphere, and there is a constant $\tilde C $ such that 
$$ \max (\|u_0 \|_{L^{\infty}(G)},  \|\nabla_E u_0 \|_{L^{\infty}(G)}) \leq \tilde C ,$$
where $\nabla_E$ denotes the standard Euclidean gradient. Then for every compact set $Z$ there is a constant ${\tilde C}(Z)$, such that for every $\delta$ and $\epsilon$ the solution of problem \eqref{eq:regMCF} satisfies 
\begin{align}\label{eq:estimateseps}
\|u^{\eps}_{\delta}(.,t)\|_{L^{\infty}(G)} & \leq \tilde{C},\\
\|\nabla_E u^{\eps}_{\delta}(.,t)\|_{L^{\infty}(G)} & \leq {\tilde C}(Z).
\end{align}
As a consequence, there exists a continuous vanishing viscosity solution $u$ of problem \eqref{eq:orgMCF} which satisfies
\begin{align}\label{eq:estimatesFromFranceschiello}
\|u(.,t)\|_{L^{\infty}(G)} & \leq \tilde{C},\\
\|\nabla_E u(.,t)\|_{L^{\infty}(G)} & \leq \tilde{C}(Z).
\end{align}
\end{theorem}

Recall that the regularized equation \eqref{eq:regMCF} 
has no critical points, hence the comparison principle established by Capogna and Citti \cite{capogna2009generalized} for viscosity solutions of the equation in the Carnot group settings is valid for \eqref{eq:regMCF} as well. 

The proof of Theorem \ref{existence} in \cite{citti2016sub} is obtained 
via an approximation, starting with vanishing viscosity solutions, in the 
sense of Definition 1. Using this explicit construction, the following weak version of the 
comparison principle follows:
\begin{theorem}\label{thm:comparisonForVanishingViscosity}
Let $G$ be a Carnot group or $\se$. 
Assume that $u$ and $v$ are vanishing viscosity solutions of \eqref{eq:orgMCF} in accordance with Definition \ref{vvisc}, identically constant out of a compact set. Suppose further

(i) For all $\xi \in G$,  $u(\xi, 0) \leq v(\xi, 0)$,

(ii) $u$ and $v$ are uniformly continuous when restricted to 
$G\times \{t = 0\}$.

Then $u(\xi, t) \leq v (\xi, t)$ for all $\xi \in G$ and $t \geq  0$. 
\end{theorem}

\begin{proof}
The solution $u^{\epsilon,R}_{\delta} $, which is defined  on the bounded cylinder, satisfies the maximum and comparison principles. 
As a consequence these properties are inherited by $u^{\epsilon}_{\delta}=\displaystyle\lim_{R\rightarrow +\infty}u^{\eps, R}_\delta
$ and $u=\displaystyle\lim_{\delta,\,\eps\rightarrow 0} u^{\eps}_{\delta}$.

\end{proof}


\section{Asymptotic behavior  of solutions}\label{sec:asymptoticProp}

In this section we establish 
the asymptotic behavior of solutions of equations \eqref{eq:orgMCF} and \eqref{eq:regMCF}, 
whose Euclidean analogue has been proved by Evans \cite{evans1991motion}, 
and extended to Carnot groups in \cite{capogna2009generalized}. 
The proof is based on a comparison with an ad hoc auxiliary function. 
Since at this stage we have the comparison principle only 
for the vanishing viscosity solutions, the geometrical results hold only for that type of solutions.

\subsection{Asymptotic behavior of the vanishing viscosity solution}

Recall that in a Carnot group of step 2 we use the notation $\xi=(x,\theta)$, 
where $x=(x_1, \cdots, x_m)$ are the variables of the first layer, 
$\theta= (\theta_1, \cdots, \theta_{n-m})$ are the variables of the second layer and in $\se$ we use $\xi=(x,y,\theta)$ with the spatial variables $x,y\in \R$ and the angular variable $\theta\in S^1$. We will write $\abs{\xi}$ in order to denote 
a pseudo-distance of $\xi$ from the origin which can describe a neighborhood of infinity. In particular 
we will denote
\begin{align}\label{eq:hDistanceDefn}
\abs{\xi}= \big( x^2+y^2 \big)^{1/2}\quad\text{in $\se$,}
\quad \abs{\xi}= \big( |x|^4 + |\theta|^2\big)^{1/4}\quad  \text{in a Carnot group.}
\end{align}

In the Euclidean and Carnot settings it is known that if the 
level sets of the initial datum are confined in some bounded region,
then the corresponding level set of the solution remains in the 
same region during the whole mean curvature flow. This result in Carnot groups can be stated by following \cite[Theorem 5.6]{capogna2009generalized} as: 
\begin{theorem}\label{thm:exteriorCylinderHeisenberg}
Assume that $G$ is a Carnot group of step 2, $u_0$ is continuous and there exists a constant $K>0$ such that
\begin{align}
\text{$u_0(\xi)$ is constant for all $\xi\in G$ satisfying $\lvert \xi \rvert\geq K.$}
\end{align}
Then there exists  $R>0$, dependent only on $K$, such that any vanishing viscosity solution of \eqref{eq:orgMCF} satisfies
\begin{align}
\text{$u(\xi,t)$ is constant for all $\xi\in G$ satisfying $\lvert \xi \rvert\geq R$ and for all $t>0$}.	
\end{align}
\end{theorem}

Here we prove the same asymptotic behavior in $\se$. Clearly, if $\xi=(x,y,\theta)\in \se$ it is sufficient to check the $(x,y)$ variables in a neighborhood of 
infinity. Precisely we prove
\begin{theorem}\label{thm:exteriorCylinder}
Assume that $u_0\in C^{\infty}(\se)$ and it is constant at the exterior of a cylinder. More precisely, assume there exists a constant $K>0$ such that
\begin{align}
\text{$u_0(\xi)$ is constant for all $\xi\in\se$ satisfying $\lvert \xi \rvert\geq K.$}
\end{align}
Then there is $R>0$, dependent only on $K$, such that the vanishing viscosity solution $u$ of \eqref{eq:orgMCF} satisfies 
\begin{align}
\text{$u(\xi,t)$ is constant for all $\xi\in\se$ satisfying $\lvert \xi \rvert\geq R$ and for all $t>0$}.
\end{align}
\end{theorem}

\begin{lemma}\label{lemma:VfunctionLem}Assume that $G=\se$. 
Let us fix $0<\delta,\eps<1 $ and let $h(\xi)=\frac{x^2 + y^2}{2}$. For all $\xi\in G,\; t>0$ 
let us call
\begin{align}\label{eq:psivu}
V^\epsilon_{\delta}(\xi, t) = \Psi\Big(h(\xi)+t\eps\Big) - 
Ct\epsilon^{1/2}\quad \text{with}\quad
\Psi(s)\equiv\bcs
 & 0 \quad\quad\quad\quad(s\geq 2),\\
 &(s-2)^3\quad (0\leq s\leq 2).
\ecs 
\end{align}
Then there exists a choice of the constant $C$ such that the function $V^{\eps}_\delta$ satisfies
\begin{align}\label{Veq}
\partial_t V^{\eps}_{\delta}-&\SUM_{i,j=1}^3\Bigg(\delta_{ij}-\frac{X_{i\delta} V^{\eps}_{\delta}X_{j\delta}V^{\eps}_{\delta}}{\eps^2+\lvert\nabla_\delta V^{\eps}_{\delta} \rvert^2}  \Bigg)(X_{i\delta}X_{j\delta})V^{\eps}_{\delta}\leq 0,
\end{align}
and at the initial time $t=0$
\begin{align}\label{eq:VFunctionConditions}
\begin{split}
\, & V^{\eps}_{\delta}(\xi,0)=0\; \text{ if\spc $h(\xi)\geq 2$},  \quad 
-1\leq V^{\eps}_{\delta}(\xi,0)\leq 0\; \text{ if\spc $1\leq h(\xi)\leq 2$}, \\
 & V^{\eps}_{\delta}(\xi,0)\leq -1\; \text{ if\spc $0\leq h(\xi)\leq 1$}.
\end{split}
\end{align}
\end{lemma}

\begin{proof}Let us first note that $\Psi\in C^{2}\big([0,\infty)\big)$,  and it satisfies
\begin{align}
\Psi^{\prime}(s)=\bcs
0\quad\quad\quad (s\geq 2),\\
3(s-2)^2\quad (0\leq s \leq 2), 
\ecs \text{ and } \quad
\Psi^{{\prime}{\prime}}(s)=\bcs
0\quad\quad\quad (s\geq 2),\\
6(s-2)\quad (0\leq s \leq 2).
\ecs
\end{align}
From this explicit expression it immediately follows that 
 $\Psi^{\prime}\geq 0$, $\Psi^{\prime\prime}\leq 0$ and $\Psi\leq 0$, $\lvert\Psi^{\prime\prime}\rvert\leq 2\sqrt{3} \,\big( \Psi^{\prime}\big)^{1/2}$,\; $\abs{\Psi^{\prime}},\abs{\Psi^{{\prime}{\prime}}}\leq 12 $,
for all $s> 0$. Besides if we call 
$$v^{\eps}_{\delta}(\xi,t)=\Psi\Big(h(\xi)+t\eps\Big),$$ 
then 
\begin{align}\label{eq:secondOrderD}
X_{i\delta}v^{\eps}_{\delta}(\xi,t)=&\Psi^{\prime}(h(\xi) + \epsilon t) X_{i\delta}h,\\
(X_{i\delta}X_{j\delta})v^{\eps}_{\delta}(\xi,t)=&\Psi^{\prime}(h(\xi) + \epsilon t) X_{i\delta}X_{j\delta}h+
\Psi^{{\prime}{\prime}}(h(\xi) + \epsilon t) X_{i\delta}hX_{j\delta}h.
\end{align}
If so we see 
\begin{align}\label{eq:mainEqnAux}
\partial_t v^{\eps}_{\delta}&-\SUM_{i,j=1}^3\Bigg(\delta_{ij}-\frac{X_{i\delta} v^{\eps}_{\delta}X_{j\delta}v^{\eps}_{\delta}}{\eps^2+\lvert\nabla_\delta v^{\eps}_{\delta} \rvert^2}  \Bigg)(X_{i\delta}X_{j\delta})v^{\eps}_{\delta}=\\
& =\eps\Psi^{\prime}-\SUM_{i,j=1}^3\Bigg(\delta_{ij}-\frac{(\Psi')^2 X_{i\delta} hX_{j\delta}h}{\eps^2+(\Psi')^2\lvert\nabla_\delta h\rvert^2}  \Bigg)(\Psi^{\prime} X_{i\delta}X_{j\delta}h+\Psi^{{\prime}{\prime}} X_{i\delta}hX_{j\delta}h)\\
& =\eps\Psi^{\prime}-\Bigg(1-\frac{(\Psi')^2 X_1 hX_1h}{\eps^2+(\Psi')^2\lvert  \nabla_\delta h\rvert^2}  \Bigg)(\Psi^{\prime} X_{1}X_{1}h+\Psi^{{\prime}{\prime}} X_{1}hX_{1}h)\\
&\quad\quad\quad -\Bigg(1-\frac{(\Psi')^2 X_{3\delta} hX_{3\delta}h}{\eps^2+(\Psi')^2\lvert  \nabla_\delta h\rvert^2}  \Bigg)(\Psi^{\prime} X_{3\delta}X_{3\delta}h+\Psi^{{\prime}{\prime}} X_{3\delta}hX_{3\delta}h)\\
&\quad\quad\quad 
+ \frac{2(\Psi')^2 (X_1 h)^2(X_{3\delta}h)^2\Psi^{\prime\prime}}{\eps^2+(\Psi')^2\lvert  \nabla_\delta h\rvert^2}    ,
\end{align}
since $X_2h=X_2X_2h=0$ and $X_1X_{3\delta}h=X_{3\delta}X_1h=0$. 
Using the fact that $\Psi^{\prime\prime}\leq 0$, $\Psi^{\prime}X_1X_1h\geq 0$, $\Psi^{\prime}X_{3\delta}X_{3\delta}h\geq 0$ we obtain 
\begin{align}
\partial_t v^{\eps}_{\delta}-&\SUM_{i,j=1}^3\Bigg(\delta_{ij}-\frac{X_{i\delta} v^{\eps}_{\delta}X_{j\delta}v^{\eps}_{\delta}}{\eps^2+\lvert\nabla_\delta v^{\eps} \rvert^2}  \Bigg)(X_{i\delta}X_{j\delta})v^{\eps}_{\delta}\leq\eps\Psi^{\prime}-\frac{\eps^2\Psi^{\prime\prime}(\lvert X_1 h\rvert^2+\lvert X_{3\delta} h\rvert^2)}{\eps^2+(\Psi')^2\abs{\nabla_{\delta}h}^2}\leq \\ &
\eps\Psi^{\prime}+\frac{\eps^2|\Psi^{\prime\prime}|\,\abs{\nabla_{\delta}h}^2}{\eps^2+(\Psi')^2\abs{\nabla_{\delta}h}^2} \leq \eps\Psi^{\prime}+\frac{2 \sqrt{3}\eps^2|\Psi^{\prime}|^{1/2}\abs{\nabla_{\delta}h}^2}{\eps^2+(\Psi')^2\abs{\nabla_{\delta}h}^2} ,
\end{align}
where we have used also the fact that $\lvert\Psi^{\prime\prime}\rvert\leq 2\sqrt{3} \,\big( \Psi^{\prime}\big)^{1/2}$. 
We can assume here that $h(\xi)<2$ (since otherwise the thesis is trivially true) and in this set $\lvert\nabla_\delta h\rvert^2 \leq 16$.

We have two cases. The first one is $\eps>\Psi^{\prime}$ resulting in
\begin{align}\label{eq:ForCConstant}
\eps\Psi^{\prime}+\frac{2 \sqrt{3}\eps^2|\Psi^{\prime}|^{1/2}\lvert\nabla_\delta h\rvert^2}{\eps^2+(\Psi')^2\lvert\nabla_\delta h\rvert^2} \leq \eps\Psi^{\prime}+
\frac{ 2\sqrt{3}\;\eps^{5/2}\lvert\nabla_\delta h\rvert^2}{\eps^2}\leq \eps^{1/2}\Psi^{\prime}+32\sqrt{3}\eps^{1/2}\leq C \eps^{1/2},
\end{align}
where $C>0$ is a fixed finite number.

In the second case where now $\eps<\Psi^{\prime}$ we find the same estimate as
\begin{align}
\eps\Psi^{\prime}+\frac{2 \sqrt{3}\eps^2|\Psi^{\prime}|^{1/2}\lvert\nabla_\delta h\rvert^2}{\eps^2+(\Psi')^2\lvert\nabla_\delta h\rvert^2}\leq \eps\Psi^{\prime}+\frac{2\sqrt{3}\eps^2(\Psi^{\prime})^{1/2}}{(\Psi^{\prime})^2}\leq \eps\Psi^{\prime}+ \frac{2\sqrt{3}\eps^2}{(\Psi^{\prime})^{3/2}}  \\\leq \eps\Psi^{\prime}+ \frac{2\sqrt{3}\eps^{2}}{\eps^{3/2}}     \leq \eps^{1/2}\Psi^{\prime}+ 2\sqrt{3}\eps^{1/2}\leq C\eps^{1/2}.
\end{align}
Choosing $C$ in this way in the definition of $v^{\eps}_{\delta}$, we immediately deduce that $V^\epsilon_{\delta}$ satisfies 
\begin{align}
\partial_t V^{\eps}_{\delta}-&\SUM_{i,j=1}^3\Bigg(\delta_{ij}-\frac{X_{i\delta} V^{\eps}_{\delta}X_{j\delta}V^{\eps}_{\delta}}{\eps^2+\lvert\nabla_\delta V^{\eps}_{\delta} \rvert^2}  \Bigg)(X_{i\delta}X_{j\delta})V^{\eps}_{\delta}\leq 0,
\end{align}
and at the initial time $t=0$
\begin{align}
\quad \quad \quad \quad &V^{\eps}_{\delta}(\xi,0)=0  &\text{if\spc $h(\xi)\geq 2$},&\quad \quad \quad \quad \quad \quad \quad \quad\\
\quad \quad \quad \quad&-1\leq V^{\eps}_{\delta}(\xi,0)\leq 0  &\text{if\spc $1\leq h(\xi)\leq 2$},&\quad \quad \quad \quad\quad \quad \quad \quad\\
\quad \quad \quad \quad&V^{\eps}_{\delta}(\xi,0)\leq -1 &\text{if\spc $0\leq h(\xi)\leq 1$}.&\quad \quad \quad \quad\quad \quad \quad \quad
\end{align}

\end{proof}

\bigskip

Now we can provide the proof of Theorem \ref{thm:exteriorCylinder}.
\begin{proof}

By Lemma \ref{lemma:VfunctionLem} we know that there exists a function $V^{\epsilon}_{\delta}$ which satisfies
\begin{align}
\partial_t V^{\eps}_{\delta}-&\SUM_{i,j=1}^3\Bigg(\delta_{ij}-\frac{X_{i\delta} V^{\eps}_{\delta}X_{j\delta}V^{\eps}_{\delta}}{\eps^2+\lvert\nabla_\delta V^{\eps}_{\delta} \rvert^2}  \Bigg)(X_{i\delta}X_{j\delta})V^{\eps}_{\delta}\leq 0,
\end{align}
and the conditions given in \eqref{eq:VFunctionConditions} at the initial time $t=0$. 
Up to a rescaling we can assume that $|u_0| \leq 1$, and $u_0 =0$ where $x^2 + y^2\geq 1$. Hence the conditions in \eqref{eq:VFunctionConditions} imply that $V^{\eps}_{\delta}(\xi,0)\leq u_0(\xi)$ for all $\xi\in \se$.
Applying the comparison principle to vanishing viscosity solutions obtained from the regularized mean curvature equation,
we deduce that $V^{\eps}_{\delta}\leq u^{\eps}_{\delta}$ in $\se\times [0,\infty)$ for each $0<\delta,\eps< 1$. This result implies
\begin{align}\label{eq:defnAtLim}
\lim_{\eps\rightarrow 0}V^{\eps}=\Psi\Big(\frac{\abs{x}^2+\abs{y}^2}{2}\Big)=0\leq u(\xi,t),
\end{align}
for all $t\geq 0$ and $\xi=(x,y,\theta)\in \se$ satisfying $(\abs{x}^2+\abs{y}^2)/2\geq 2$. Hence $u\geq 0$ if $(\abs{x}^2+\abs{y}^2)/2\geq 2$. Arguing in the same way with the function 
$\tilde{V}^{\eps}_{\delta}=-V^{\eps}_{\delta}$, we deduce
\begin{align}\label{eq:maxPrincipleGives1}
u\leq 0\quad\text{if}\quad (\abs{x}^2+\abs{y}^2)/2\geq 2.
\end{align}
Hence \eqref{eq:defnAtLim} and \eqref{eq:maxPrincipleGives1} give
$ u(\xi,t)= 0 $ for all $t\in[0,T]$ and $\xi=(x,y,\theta) $ such that $(\abs{x}^2+\abs{y}^2)/2\geq 2$.
\end{proof}

\subsection{Asymptotic behavior of the approximating solutions}

Here we prove that the solutions of problem \eqref{eq:regMCF}, which are obtained as limits of the solutions of problem \eqref{eq:regMCFbounded}, exponentially tend to a constant value.

\begin{theorem}\label{thm:exponentialEstimateThmHeisenberg} Let $G$ be either a Carnot group of step 2 or $\se$, and let $u^\eps_{\delta}$ be the solution   of problem \eqref{eq:regMCF}, obtained as the limit of the solution of problem \eqref{eq:regMCFbounded}. 
For each $0<\delta,\eps<1$ there exist some finite numbers $B,b>0$ independent of $\eps$ and $\delta$ such that
\begin{equation}
\abss{1-u^{\eps}_{\delta}(\xi,t)}\leq B\expp^{-b|\xi|}\quad \text{for all $\xi\in G\times[0,T]$},
\end{equation}
where $T$ is a positive finite number denoting the final time and $\abs{\xi}$ is defined in \eqref{eq:hDistanceDefn}.
\end{theorem}

\begin{proof}
The proof is based on comparing the function $1-u^{\eps}_{\delta}$ with the auxiliary function 
\begin{equation}v^\eps_{\delta}(\xi,t) = \Psi(h(\xi))\quad \text{where}
\quad \Psi(s)=\hat{c}\expp^{-\sigma(2T-\alpha t)s}, 
\end{equation} 
with $0<\alpha<1$,\; $0<\sigma<\infty$ and constant $\hat{c}=2\expp^{4\sigma T}$. The function $h$ will be a polynomial, and we will need to perform different choices of this function for the first and the second layer of the 
Carnot group and for the $\se$ group. As before we assume that 
\begin{equation}\label{eq:u0ConstantAssumption}
\abs{u_0}\leq 1\;\text{on $G$},\quad u_0=0\quad\text{for all}\quad |\xi|\geq 1.
\end{equation}

Now we proceed by using a procedure similar to the one in the proof of \cite[Theorem 5.1]{bonfiglioli2002uniform}, and we show that both in Carnot groups of step 2 and in $\se$ we have 
\begin{equation}\label{inequality}
\partial_t v^{\eps}_{\delta}-\SUM_{i,j=1}^n\Bigg(\delta_{ij}- \frac{X_{i\delta} v^{\eps}_{\delta}X_{j\delta} v^{\eps}_{\delta}}{\eps^2+\lvert \nabla_\delta v^{\eps}_{\delta}\rvert^2} \Bigg) (X_{i\delta}X_{j\delta})v^{\eps}_{\delta}\geq 0.
\end{equation}

In a Carnot group of step 2 we will consider a function $h(\xi) =|x|^2 + \SUM_{s=1}^{n-m} \sqrt{1 + \theta_s^2} $ where $x$ denotes the horizontal and $\theta$ the vertical variables. Furthermore we will express the horizontal vector fields $X_i$ by using \eqref{eq:VFs_representation_carnot}, and the approximating vertical vector fields $X_{i\delta}$ by using \eqref{eq:vertical_vfs}. In this case we write all the vector fields as:
$$X_j h= 2  x_j + \SUM_{k,s} w_{jk}^{(s)}x_k\frac{\theta_s}{\sqrt{1 + \theta_s^2}},\quad j=1,\dots,m,
$$
$$
 X_i X_jh= 2 \delta_{ij} + \SUM_{s}w_{ji}^{(s)}\frac{\theta_s}{\sqrt{1 + \theta_s^2}} +  \SUM_{p,k,s} w_{ip}^{(s)}w_{jk}^{(s)}\frac{ x_kx_p}{(1 + \theta_s^2)^{3/2}},\quad i,j=1,\dots,m,
$$
$$
X_{j\delta}h=\frac{\delta \theta_{j-m}}{\sqrt{1+\theta_{j-m}^2}},\quad j=m+1,\dots,n,
$$
$$
X_{i\delta}X_{j\delta}h=\frac{\delta_{ij}\delta^2}{(1+\theta_{j-m}^2)^{3/2}}, \quad i,j=m+1,\dots,n,
$$
$$
X_{i\delta}X_{j}h=\frac{\delta}{(1+\theta_{i-m}^2)^{3/2}}\SUM_{k}w_{jk}^{(i)}x_k=X_jX_{i\delta}h,\quad i=m+1,\dots,n,\quad j=1,\dots,m.
$$

We use the derivatives of $v^{\eps}_{\delta}$ computed in \eqref{eq:secondOrderD} and the fact that $X_iX_{j\delta}h=X_{j\delta}X_i h$, then we write the curvature operator on the function 
$v^{\eps}_{\delta}$ as:
\begin{align}
\partial_t v^{\eps}_{\delta}  & -\SUM_{i,j=1}^n\Bigg(\delta_{ij}- \frac{X_{i\delta} v^{\eps}_{\delta}X_{j\delta}v^{\eps}_{\delta}}{\eps^2+\lvert \nabla_\delta v^{\eps}_{\delta}\rvert^2} \Bigg) (X_{i\delta}X_{j\delta})v^{\eps}_{\delta}\\ 
= & \partial_t v^{\eps}_{\delta}
-\SUM_{i,j=1}^m\Bigg(\delta_{ij}-\frac{(\Psi')^2 X_i hX_jh}{\eps^2+(\Psi')^2\lvert\nabla_\delta h\rvert^2}  \Bigg)
(\Psi^{\prime} X_{i}X_{j}h+\Psi^{{\prime}{\prime}} X_{i}hX_{j}h)\\
& -2\SUM_{i=1}^m\SUM_{j=m+1}^{n}\Bigg(-\frac{(\Psi')^2 X_i hX_{j\delta}h}{\eps^2+(\Psi')^2\lvert\nabla_\delta h\rvert^2}  \Bigg)
(\Psi^{\prime} X_{i}X_{j\delta}h+\Psi^{{\prime}{\prime}} X_{i}hX_{j\delta}h)\\
& -\SUM_{i,j=m+1}^{n}\Bigg(\delta_{ij}-\frac{(\Psi')^2 X_{i\delta} hX_{j\delta}h}{\eps^2+(\Psi')^2\lvert\nabla_\delta h\rvert^2}  \Bigg)
(\Psi^{\prime} X_{i\delta}X_{j\delta}h+\Psi^{{\prime}{\prime}} X_{i\delta}hX_{j\delta}h)\\
= & \partial_t v_{\delta}^{\eps}-Q_1-Q_2-Q_3. 
\end{align}

Let us first consider $Q_1$. Due to the symmetry of $\delta_{ij}-\frac{(\Psi')^2 X_i hX_jh}{\eps^2+(\Psi')^2\lvert\nabla_\delta h\rvert^2}$ and the antisymmetry of $w_{ji}^{(s)}$ (note that $w_{ji}(s)'s$ correspond to a skew-symmetric matrix, see \eqref{eq:VFs_representation_carnot} and \eqref{eq:skewsymmetric_W})
we deduce that 
$$
-\SUM_{i,j=1}^m\Bigg(\delta_{ij}-\frac{(\Psi')^2 X_i hX_jh}{\eps^2+(\Psi')^2\lvert\nabla_\delta h\rvert^2}  \Bigg)
\Psi^{\prime}\Big( 2\delta_{ij} + \SUM_s w_{ji}^{(s)}\frac{\theta_s}{\sqrt{1 + \theta_s^2}} +  \SUM_{p,k,s} w_{ip}^{(s)}w_{jk}^{(s)}\frac{ x_kx_p}{(1 + \theta_s^2)^{3/2}}  \Big)= $$
$$
=-\Psi^{\prime}\SUM_{i,j=1}^m\Bigg(\delta_{ij}-\frac{(\Psi')^2 X_i hX_jh}{\eps^2+(\Psi')^2\lvert\nabla_\delta h\rvert^2}  \Bigg)\Big( 2\delta_{ij} +
 \SUM_{p,k,s} w_{ip}^{(s)}w_{jk}^{(s)}\frac{ x_kx_p}{(1 + \theta_s^2)^{3/2}}\Big)  \leq 0, $$
where we use also that $\Psi^{\prime}<0$ and the matrix of coefficients of the equation is positive semidefinite. On the other hand we see by using $\Psi^{\prime\prime}>0$ that
\begin{align}
\begin{split}
\SUM_{i=1}^m & \Bigg(\delta_{ij}-\frac{(\Psi')^2 (X_i h)^2}{\eps^2+(\Psi')^2\lvert\nabla_\delta h\rvert^2}  \Bigg) \Psi^{{\prime}{\prime}} X_{i}h X_{j}h=\SUM_{\substack{i,j=1 \\ i\neq j}}^m \Bigg(-\frac{(\Psi')^2 X_i h X_j h}{\eps^2+(\Psi')^2\lvert\nabla_\delta h\rvert^2}  \Bigg) \Psi^{{\prime}{\prime}} X_{i}hX_{j}h\\
& +\SUM_{i=1}^m\Bigg(1-\frac{(\Psi')^2 (X_i h)^2}{\eps^2+(\Psi')^2\lvert\nabla_\delta h\rvert^2}  \Bigg) \Psi^{{\prime}{\prime}} (X_{i}h)^2\\
& \leq \SUM_{i=1}^m  \Bigg|1-\frac{(\Psi')^2 (X_i h)^2}{\eps^2+(\Psi')^2\lvert\nabla_\delta h\rvert^2}  \Bigg| \Psi^{{\prime}{\prime}} (X_{i}h)^2 \leq \Psi^{{\prime}{\prime}}  \SUM_{i=1}^m |X_i h|^2\leq Z_1 \Psi^{{\prime}{\prime}} |x|^2.
\end{split}
\end{align}
Therefore
\begin{equation}
Q_1\leq  Z_1 \Psi^{{\prime}{\prime}} |x|^2\leq Z_1 \sigma^2 (2T-\alpha t)^2 \Psi \abs{x}^2,
\end{equation}
where $Z_1>0$ is a fixed finite number.

We continue with $Q_2$ and thanks to the antisymmetry of $w_{ik}^{(j)}$ we find that
\begin{align}
\begin{split}
Q_2 & \leq  2 \SUM_{i=1}^m \SUM_{j=m+1}^{m}  \Bigg|\frac{(\Psi')^2 X_i h X_{j\delta} h}{\eps^2+(\Psi')^2\lvert\nabla_\delta h\rvert^2}  \Bigg|\, \abs{\Psi^{\prime} X_{i}X_{j\delta}h}\leq \SUM_{j=m+1}^{n}\SUM_{i=1}^m \abs{\Psi^{\prime}X_iX_{j\delta}h}
\\&  \leq 2 \SUM_{i=1}^m \SUM_{j=m+1}^{n} \abs{\Psi^{\prime}}\,\abs{\frac{\delta}{(1+\theta_{j-m}^2)^{3/2}}\SUM_{k}w_{ik}^{(j)}x_k}\leq 2\delta \abs{x}\abs{\Psi^{\prime}}\leq \delta(\abs{x}^2+1)\sigma\alpha\Psi.
\end{split}
\end{align}

Finally we consider $Q_3$. Note that
\begin{align}
\begin{split}
Q_3\leq & \Bigg|\SUM_{i=m+1}^n\Bigg(\delta_{ii}-\frac{(\Psi')^2 (X_{i\delta} h)^2}{\eps^2+(\Psi')^2\lvert\nabla_\delta h\rvert^2}  \Bigg) \Psi^{{\prime}{\prime}} (X_{i\delta}h)^2-\SUM_{\substack{i,j=m+1 \\ i\neq j}}^{n}\Bigg(\frac{(\Psi')^2 (X_{i\delta} h)^2(X_{j\delta} h)^2}{\eps^2+(\Psi')^2\lvert\nabla_\delta h\rvert^2}  \Bigg) \Psi^{{\prime}{\prime}}\Bigg|
\\&
\leq \Bigg|\SUM_{i=m+1}^n\Bigg(\delta_{ii}-\frac{(\Psi')^2 (X_{i\delta} h)^2}{\eps^2+(\Psi')^2\lvert\nabla_\delta h\rvert^2}  \Bigg) \Psi^{{\prime}{\prime}} (X_{i\delta}h)^2\Bigg|\leq  \SUM_{i=m+1}^{n}\Psi^{\prime\prime}\frac{\delta^2 \theta_{i-m}^2}{1+\theta_{i-m}^2}\leq Z_3 \sigma^2(2T-\alpha t)^2 \Psi,
\end{split}
\end{align}
where $Z_3>0$ is a fixed finite number.

As a consequence 
\begin{align}
\begin{split}
\partial_t v^{\eps}  & -Q_1-Q_2-Q_3\geq\\
&\geq \sigma  \alpha\Psi(|x|^2+1)-Z_1 \sigma^2 (2T-\alpha t)^2 \Psi \abs{x}^2-\delta(\abs{x}^2+1)\sigma\alpha\Psi-Z_3 \sigma^2(2T-\alpha t)^2 \Psi
\geq 0,
\end{split}
\end{align}
if $\sigma$ and $\delta$ are small.


We choose for the $\se$ setting $h(\xi) = \frac{x^2 + y^2}{2}$. Then by using $X_2 h=0$, $\Psi^{\prime}X_1X_1 h<0$, $\Psi^{\prime}X_{3\delta}X_{3\delta} h<0$ and $\Psi^{\prime\prime}>0$ we obtain
\begin{align}\label{eq:soMaxPrinciple}
\begin{split}
\partial_t v^{\eps}_{\delta}   -\SUM_{i,j=1}^3\Bigg(\delta_{ij}- \frac{X_{i\delta} v^{\eps}_{\delta}X_{j\delta}v^{\eps}_{\delta}}{\eps^2+\lvert \nabla_\delta v^{\eps}_{\delta}\rvert^2} \Bigg) (X_{i\delta}X_{j\delta})v^{\eps}_{\delta}
\geq &   \partial_t v^{\eps}_{\delta} - 
\SUM_{i=1}^3 \frac{\eps^2}{\eps^2+(\Psi')^2\lvert\nabla_\delta h\rvert^2}\Psi^{{\prime}{\prime}} (X_{i\delta}h)^2\\
\geq & \sigma  \alpha\Psi|x|^2,
\end{split}
\end{align}
and we conclude as before that \eqref{inequality} is satisfied. Note that  
$$\partial_t(1-u^{\eps}_{\delta})-\SUM_{i,j=1}^3\Bigg(\delta_{ij}- \frac{X_{i\delta} (1-u^{\eps}_{\delta})X_{j\delta}(1-u^{\eps}_{\delta})}{\eps^2+\lvert\nabla_\delta (1-u^{\eps}_{\delta}) \rvert^2}  \Bigg)(X_{i\delta}X_{j\delta})(1-u^{\eps}_{\delta})=0.$$
Recall that $0\leq \abs{1-u_0}\leq 2$ due to \eqref{eq:u0ConstantAssumption}.
Hence, applying the comparison principle for vanishing viscosity solutions (see Theorem \ref{thm:comparisonForVanishingViscosity} and \cite{citti2016sub} for more details) to the functions $1-u^{\eps}_{\delta}$ (as well as $u^{\eps}_{\delta}-1$) and $v^{\eps}_{\delta}$ we find 
\begin{align}\label{eq:finalResultExpEst}
\abs{1-u^{\eps}_{\delta}}\leq v^{\eps}_{\delta}\quad\text{in $G\times [0,T]$}.
\end{align}
The proof is complete. 
\end{proof}

\section{Viscosity and vanishing viscosity solutions}\label{sec:viscosityUniqueness}

In this section we prove Theorem \ref{thm:univiscosityTheorem}, which says that any viscosity solution is a limit of a family of special vanishing viscosity solutions obtained with 
the procedure of \cite{deckelnick2000error}. We provide the proof in the generic setting of Carnot groups of step 2. We will denote any Carnot group of such family by $G\simeq \R^n$ while we will use the standard notation with $\se$ for the rototranslation group. 

Let us now define an exponential distance in the considered groups in terms of the exponential increments given in \eqref{eq:carnot_increment_exp} and \eqref{eq:se_increment_exp}. We fix $\beta$ in such a way that $0<\beta=\delta\leq \eps<1$. Then we define
\begin{equation}\label{eq:distance_functions}
 d_\beta(\xi, \eta) := \Big(e_1^2 +\dots+e_{m}^2+ \beta^2 e_{m+1}^2+\dots+ \beta^2e_{n}^2   \Big)^{1/2},
\end{equation}
and
\begin{equation}\label{eq:d0_d3_defn}
 d_0^2(\xi,\eta):=e_1^2+\dots+e_m^2,\quad d_3^2(\xi,\eta):=\beta^2 e_{m+1}^2+\dots +\beta^2 e_{n}^2.
 \end{equation} 

We fix an initial datum  $u_0$ and denote by $\operatorname{Lip}(u_0)$ its Lipschitz constant with respect to  $d_\beta$. We will denote by $\alpha$, $\gamma$, $M$ the parameters to be specified later, but satisfying
\begin{align}\label{eq:muExpression}
\gamma >2,\quad 0<\beta<1, \quad \mu=\frac{\gamma 4^{\gamma}\operatorname{Lip}(u_0)^{\gamma}}{ M^{\gamma-1}}. 
\end{align}

\begin{remark}
Assume that $\beta$ is fixed. The distance $d_{\beta}(\xi,\eta)$ depends on $\beta$, therefore, so $\operatorname{Lip}(u_0)$ does. In order to remove the dependency of the Lipschitz constant one may perform the following substitutions:
\begin{equation}\label{eq:substitution_u_0}
U(\xi,t)=u(\beta \xi,t),\quad\quad U^{\eps}_{\delta}(\eta)=u^{\eps}_{\delta}(\beta \eta,t),
\end{equation}
and
$$U_0(\xi)=u_0(\beta \xi),\quad\quad U^{\eps}_{0,\delta}(\eta)=u_0(\beta \eta).$$
If we denote by $\overline{\operatorname{Lip}}(u_0)$ the Lipschitz constant independent of $\beta$, that is,
$$\displaystyle\overline{\operatorname{Lip}}(u_0)=\underset{\xi,\eta\in G}{\operatorname{max}}\frac{\abs{u_0(\xi)-u_0(\eta)}}{d_1(\xi,\eta)}=\frac{\abs{u_0(\tilde{\xi})-u_0(\tilde{\eta})}}{d_1(\tilde{\xi},\tilde{\eta})}.$$
We can employ the substitutions given by \eqref{eq:substitution_u_0} and write
$$\displaystyle\operatorname{Lip}(U_0)\leq \frac{\overline{\operatorname{Lip}}(U_0)}{\beta}=\frac{ d_{1}(\beta\tilde{\xi},\beta\tilde{\eta})\overline{\operatorname{Lip}}(u_0)}{\beta d_{1}(\tilde{\xi},\tilde{\eta})}=\frac{ \beta d_{1}(\tilde{\xi},\tilde{\eta})\overline{\operatorname{Lip}}(u_0)}{\beta d_{1}(\tilde{\xi},\tilde{\eta})}=\overline{\operatorname{Lip}}(u_0).$$
Consequently we have the Lipschitz constant of the dilated initial condition, now which is independent of $\beta$ and in the $\se$ case this independence is valid for any $\theta_0$ (see \eqref{eq:theta_0_condition}). Furthermore, if \eqref{eq:contradHypth} is satisfied by the functions $u$ and $u^{\eps}_\delta$, then the dilated versions of those functions satisfy the same inequality with the same constants. Therefore, the solutions from now on can be considered as the dilated versions $U(\xi,t)$ and $U^{\eps}_{\delta}(\eta,t)$. 
\end{remark}

Then we introduce the function
\begin{equation}\label{eq:testFunctionDefnH}
\phi(\xi,\eta,t)=\frac{\mu}{\gamma}\eps^{1-\frac{\gamma}{2}}d_\beta^{\gamma}(\xi, \eta)+\frac{Mt}{2T}\eps^{\alpha}. 
\end{equation}

If  $u$  and $u^\eps_{\delta}$ are the solutions of \eqref{eq:orgMCF} and \eqref{eq:regMCF}, respectively (note that $u$ is continuous while $u^{\eps}_{\delta}$ is smooth), with the same 
initial condition $u(\xi, 0)=u^\eps_{\delta}(\xi, 0)=u_0(\xi)$, we write our test function as follows: 
\begin{equation}\label{eq:defnOmega2}
\om(\xi,\eta,  t)=u(  \xi,  t)-u^\eps_{\delta}( \eta,  t)-\phi( \xi,\eta,t),
\end{equation}
with suitable constants $M\geq 0$ and $0<\eps<1$.

\begin{lemma}\label{lemma:lem41}
Let $u$ and $u^{\eps}_{\delta}$ be two continuous functions such that $u(\xi, 0)=u^{\eps}_{\delta}(\xi, 0), $ for every $\xi \in G$.

Let $d_\beta$ be the distance defined in  \eqref{eq:distance_functions} and $\abs{\xi},\,\abs{\eta}$ be defined as in \eqref{eq:hDistanceDefn}. Let $\omega$ be the test function defined in \eqref{eq:defnOmega2}.
Assume that 
\begin{itemize}
\item{
There exist constants $B_{\eps}$, $b_{\eps}$ $R_\eps>0$ such that for every $|\xi|>R_\eps$
$$
|1- u(\xi, t)|\leq B_{\eps}\expp^{-b_{\eps}\abs{\xi}^2}, \quad 
|1- u^{\eps}_{\delta}(\xi, t)|\leq B_{\eps}\expp^{-b_{\eps}\abs{\xi}^2},
$$}
\item{
There exists a constant $\tilde C>0$, such that for every $\xi\in G$, $t\geq 0$
$$|u(\xi, t)| \leq\tilde C, \quad \quad |u^{\eps}_{\delta}(\xi, t)| \leq\tilde C.$$ }
\item{
\begin{equation}\label{eq:contradHypth} 
\sup_{\xi\in G, 0< t\leq T} (u-u^{\eps}_{\delta})(\xi,t) 
> 
M\eps^{\alpha}.\end{equation}}

\end{itemize}
Then 
$$\sup_{\xi,\eta\in G, \atop 0\leq t\leq T}\om(\xi,\eta, t)= \sup_{d_\beta(\xi, \eta) \leq r,\; |\xi|, |\eta|\leq R_{\eps} \atop 0< t\leq T}\om(\xi, \eta, t)$$
where $ r:= \Big( \frac{2\gamma \tilde{C}}{\mu}\eps^{\frac{\gamma}{2}-1}-\frac{M\gamma}{4\mu}\eps^{\alpha+\frac{\gamma}{2}-1} \Big)^{\frac{1}{\gamma}}$ and
$\tilde{R}_{\eps}:=r+\sqrt{\frac{4 B_{\eps}}{M b_{\eps}}\eps^{-\alpha}}$, and the maximum is attained at a point $(\hat{\xi},\hat{\eta},\hat{t})$ with $\hat{t}>0$.

\end{lemma}

\begin{proof}
 
We first show that the supremum on the whole space is equal to the supremum 
under the condition that $d_{\beta}(\xi,\eta) \leq r$. Indeed using \eqref{eq:defnOmega2} and \eqref{eq:contradHypth} we observe 
\begin{align}\label{eq:bound1Onom}
\sup_{\xi,  \eta\in G, \atop 0\leq  t\leq T}\om(\xi,\eta, t)&
\geq M\eps^{\alpha}-\frac{M}{2T}\eps^{\alpha}T=\frac{M}{2}\eps^{\alpha} .
\end{align}
Note that when $d_{\beta}(\xi,\eta)\geq r$ we have
\begin{align}\label{eq:boundsContradicted}
\om(\xi,\eta, t)\leq &\sup_{ \xi\in G,\,0\leq t\leq T} u (\xi, t)+\sup_{ \eta\in G,\,0\leq t\leq T}\Big(-u^{\eps}_{\delta}(\eta,t)\Big)-\frac{\mu}{\gamma}\eps^{1-\frac{\gamma}{2}}d_\beta(\xi, \eta)^{\gamma}\\
\leq & 2 \tilde{C}-\frac{\mu}{\gamma}\eps^{1-\frac{\gamma}{2}}d_\beta(\xi, \eta)^{\gamma}\leq 
\frac{M}{4}\eps^{\alpha}, 
\end{align}
due to the choice of $r$. 
Consequently we deduce that
\begin{align}
\sup_{\xi,  \eta, \in G,\atop 0\leq t\leq T} \om(\xi,\eta, t)=
\sup_{\substack{\text{$\xi, \eta\in G,\,$  $d_{\beta}(\xi, \eta)\leq r,$}\\ \;0\leq t\leq T}}
\om(\xi,\eta,t).
\end{align}

Now we show that the supremum is not achieved in a neighborhood of infinity. 
Indeed for $\abs{\eta}\geq \tilde{R}_{\eps}$, 
we see by the first assumption that
\begin{align}\label{eq:estForEtaSoln}
\abs{1-u^{\eps}_{\delta}(\eta, t)}&\leq B_{\eps}\expp^{-b_{\eps}\abs{\eta}^2}\leq \frac{B_{\eps}}{b_{\eps}\abs{\eta}^2}   \leq  \frac{M}{8}\eps^{\alpha}.
\end{align}
Analogously 
\begin{align}\label{eq:estForXiSoln}
\abs{1-u(\xi, t)}\leq \frac{M}{8}\eps^{\alpha},
\end{align}
for $\abs{\xi}\geq \tilde{R}_{\eps}$. If
$\abs{\eta}$, $\abs{\xi}\geq \tilde{R}_{\eps}$, then it follows that  
\begin{align}\label{eq:expEstimateForAll}
\om(\xi, \eta, t) & \leq \abss{1-u(\xi, t)}+
\abss{1-u^{\eps}_{\delta}(\eta,t)}\leq \frac{M}{4}\eps^{\alpha}.
\end{align}
Therefore we deduce 
\begin{align}
\sup_{\xi, \eta\in G,\atop \; 0\leq t\leq T}\om(\xi,\eta,t)=
\sup_{\substack{\text{$\xi, \eta \in G,$ }\\  \abs{\xi}\leq \tilde{R}_{\eps},\;\abs{\eta}\leq \tilde{R}_{\eps},\\ 
0\leq t\leq T}}\om(\xi,\eta, t).
\end{align}
We will denote the point 
where $\omega$ is maximum by $(\hat{\xi},\hat{\eta},\hat{t})$. 

Finally we see that $\hat{t}>0$ must hold. Observe for $t=0$ that
\begin{align}\label{eq:ForZeroTContradiction}
\om(\xi,\eta,0)=u_0(\xi)-u_0(\eta)-\frac{\mu}{\gamma}\eps^{1-\frac{\gamma}{2}}d_\beta(\xi,\eta)^{\gamma}\\
\leq d_\beta(\xi,\eta)\operatorname{Lip}(u_0) \Big(1-\frac{\mu}{\operatorname{Lip}(u_0)\gamma }\eps^{1-\frac{\gamma}{2}}d_\beta(\xi,\eta)^{\gamma-1}\Big).
\end{align}

Consider the first case with $d_\beta(\xi,\eta)\operatorname{Lip}(u_0)\leq \frac{M\eps^{\alpha}}{4}$. Then from \eqref{eq:ForZeroTContradiction} we deduce
\begin{align}
\om(\xi,\eta, 0)\leq \frac{M\eps^{\alpha}}{4},
\end{align}
which together with \eqref{eq:contradHypth} ensures that the maximum is not achieved at $t=0$. 

Now consider the second case where $d_\beta(\xi,\eta)\operatorname{Lip}(u_0)> \frac{M\eps^{\alpha}}{4}$. Then we observe that \eqref{eq:muExpression} and \eqref{eq:ForZeroTContradiction} give
\begin{align}\label{eq:tPositiveContradiction}
\om(\xi,\eta, 0)\leq d_\beta(\xi,\eta)\operatorname{Lip}(u_0) \Big(1-\frac{\mu }{\operatorname{Lip}(u_0)\gamma }\eps^{1-\frac{\gamma}{2}}d_\beta(\xi,\eta)^{\gamma-1}\Big)\leq 0,
\end{align}
which contradicts \eqref{eq:bound1Onom} and ensures that $\hat{t}>0$.
\end{proof}

At this point we can determine $\theta_0$, which appears in \eqref{eq:se_increment_exp}. In the $\se$ setting, we will fix 
\begin{equation}\label{eq:theta_0_condition}
\theta_0=\frac{\hat{\theta}_{\xi}+\hat{\theta}_{\eta}}{2},
\end{equation}
so that the orientation components of the points, $\hat{\xi}=(\hat{x}_{\xi},\hat{y}_{\xi},\hat{\theta}_{\xi})$ and $\hat{\eta}=(\hat{x}_{\eta},\hat{y}_{\eta},\hat{\theta}_{\eta})$, which are the points maximizing a test function $\omega(\xi,\eta, t),$ satisfy
\begin{equation}\label{eq:symmetry_trick}
\hat{\theta}_{\xi}-\theta_0=\hat{\theta}_{\xi}-\frac{\hat{\theta}_{\xi}+\hat{\theta}_{\eta}}{2} =
\frac{\hat{\theta}_{\xi}-\hat{\theta}_{\eta}}{2} =
-(\hat{\theta}_\eta-\theta_0).
\end{equation}

In order to simplify the computations in the proof of Theorem \ref{thm:univiscosityTheorem} we state here some properties of the 
derivatives of the function $\phi$. We will denote the vector fields by $X_{i\delta}^\xi$ and $X_{i\delta}^\eta$ while 
we take the derivatives with respect to $\xi$ and $\eta$, respectively. For the sake of simplicity we will write $\phi$ instead of $\phi(\hat{\xi},\hat{\eta},\hat{t})$ and $d_{\beta}$, $d_{0}$, $d_{3}$ instead of $d_{\beta}(\hat{\xi},\hat{\eta})$, $d_{0}(\hat{\xi},\hat{\eta})$, $d_{3}(\hat{\xi},\hat{\eta})$ from now on, as long as explicit notation is not specifically required. Furthermore we will use the notation $\tilde{K}$ in order denote a fixed finite positive number which is not necessarily the same each time it appears.

%

Let us first collect some properties of the first derivatives of the function $\phi$. 
\begin{lemma}\label{rem:fistOrderEst_both}Assume that $d_\beta\leq 1$. 
At the point $(\hat{\xi},\hat{\eta},\hat{t})$, if $\operatorname{deg}i=1$, we have
\begin{equation}\label{firstderivphi}
 |X_{i}^{\xi}\phi|,\,|X_{i}^{\eta}\phi|\leq  \tilde{K}\mu  \eps^{1-\frac{\gamma}2} d_\beta^{\gamma-2}d_0, 
\quad |X^\xi_i \phi+ X^\eta_i\phi|\leq \tilde{K}\mu\eps^{1-\frac{\gamma}2} d_\beta^{\gamma-1}d_0\beta,
\end{equation}
if $\operatorname{deg}i=2$
\begin{equation}\label{remarkderivative}
|X_{i\delta}^{\xi}\phi|,\,|X_{i\delta}^{\eta}\phi|\leq  \tilde{K}\mu  \eps^{1-\frac{\gamma}2}d_\beta^{\gamma-2}d_0\delta(d_\beta + \beta),\quad
|X^\xi_{i\delta} \phi+ X^\eta_{i\delta}\phi|\leq \tilde{K}\mu\eps^{1-\frac{\gamma}2} d_\beta^{\gamma}\delta.
\end{equation}

Finally, for $d_0$ and $d_3$ defined in \eqref{eq:d0_d3_defn}, the following estimate of the gradient holds
\begin{equation}\label{remarknabla}
|\nabla_0 ^{\zeta} \phi| \geq \mu\epsilon^{1-\frac{\gamma}{2}} d^{\gamma -2}_\beta d_0 , \quad |\nabla^\eta_{\delta} \phi |^2\geq \gamma^2 \mu^2\eps^{2-\gamma}(d_0^{2\gamma-2}+\beta^2\delta^2d_3^{2\gamma-2}). \end{equation}
\end{lemma}

\begin{proof}
The proof is a direct computation, which is similar in Carnot groups and in $\se$. 
We first show if $\operatorname{deg}i=\operatorname{deg}j=1$ and $|\theta_\xi- \theta_\eta|\leq \pi/4$, that 
\begin{equation}\label{stima11}
|X_{ i}^{\xi}e_j|, |X_{ i}^{\eta}e_j|\leq 1, \quad 
|X_{ i}^{\xi}e_j| \geq  \delta_{ij}/2,
\quad |X_{ i}^{\xi}e_j + X_{ i}^{\eta}e_j|=0.
\end{equation}

Indeed, using the expression \eqref{eq:carnot_increment_exp}, we obtain in Carnot groups 
$$
X_{ i}^{\xi}e_j = \delta_{ij},\quad  X_{ i}^{\eta}e_j = - \delta_{ij}\quad
\text{ if }
\operatorname{deg}i=\operatorname{deg}j=1.$$
In $\se$, we use the increments given by \eqref{eq:se_increment_exp} and find
\begin{align}
X^\xi_1 e_1= & \cos(\frac{\hat{\theta}_{\xi}-\hat{\theta}_{\eta}}{2}), & X^\xi_1 e_2=X^\xi_2 e_1=0, \quad\quad\quad\quad\quad   & X^\xi_2 e_2 = 1 , \\
X^\eta_1 e_1 = & -\cos(\frac{\hat{\theta}_{\xi}-\hat{\theta}_{\eta}}{2}), & X^\eta_1 e_2=X^\eta_2 e_1=0, \quad\quad\quad\quad\quad  & X^\eta_2 e_2 =-1.
\end{align}

When $\operatorname{deg}i=1,\;\operatorname{deg}j=2,$ we have 
\begin{equation}\label{stima12}
|X_{ i}^{\xi}e_j|, |X_{ i}^{\eta}e_j|\leq  d_\beta,\quad |X_{ i}^{\xi}e_j + X_{ i}^{\eta}e_j|\leq d_\beta,
\end{equation}
where we use as structure constants for $\se$ the constants of the Heisenberg group. 
Indeed we obtain in Carnot groups 
$$X_{ i}^{\xi}e_j  =\SUM_{l=1}^{m}w_{il}^{(j-m)}e_l = 
 X_{ i}^{\eta}e_j,$$
and in $\se $ we find: 
$$X^\xi_1 e_3= X^\eta_1 e_3=\sin(\frac{\hat{\theta}_{\xi}-\hat{\theta}_{\eta}}{2} ) = \frac{e_2}{2}+O(\abs{e_2}^3), \quad \quad X^\xi_2 e_3 = X^\eta_2 e_3 = 0. $$

From here and the expression of the function $\phi$ and the distance function $d_\beta$ 
in \eqref{eq:distance_functions} we obtain
\begin{equation}\label{derivaphi}
X_{ i\delta}^{\zeta} \phi = \mu\epsilon^{1-\frac{\gamma}{2}} d^{\gamma -2}_\beta
\Big(\sum_{j=1 }^m e_j X_{ i\delta}^{\zeta}e_j +  \beta^2 \sum_{j=m+1 }^n e_j X_{i\delta}^{\zeta}e_j\Big),
\end{equation}
where $\zeta$ can be either $\xi$ or $\eta$. 
Applying \eqref{stima11} and \eqref{stima12} we obtain 
$$ 
|X_{ i\delta}^{\zeta} \phi |\leq  \tilde{K}\mu\epsilon^{1-\frac{\gamma}{2}} d^{\gamma -2}_\beta
\Big(d_0 +  \beta d_0 d_3\Big),
$$
and using the assumption 
$d_\beta\leq 1$, the first inequality in \eqref{firstderivphi} for $\deg i=1$ follows. 
Moreover, using the last equality in \eqref{stima11}, we find

\begin{align}\label{sommaderivaphi}
\begin{split}
(X_{ i}^{\xi} + X_{ i}^{\eta}) \phi = & 2\gamma\mu\epsilon^{1-\frac{\gamma}{2}} d^{\gamma -2}_\beta \beta^2 \sum_{j=m+1 }^n e_j  \SUM_{l=1}^{m}w_{il}^{(j-m)}e_l\quad\quad\text{in Carnot groups},\\
(X_{ i}^{\xi} + X_{ i}^{\eta}) \phi = &  \gamma\mu\epsilon^{1-\frac{\gamma}{2}} d^{\gamma -2}_\beta \delta_{i1}\beta^2 e_3 \sin(\frac{\hat{\theta}_{\xi}-\hat{\theta}_{\eta}}{2} )\quad\quad\text{in $\se$}.
\end{split}
\end{align}

By using \eqref{derivaphi}, and the fact that $|X_{ i\delta}^{\zeta}e_j| \geq \delta_{ij}/2 $ we find for $\deg i=1$, 
$$
|X_{ i}^{\zeta} \phi| = \mu\epsilon^{1-\frac{\gamma}{2}} d^{\gamma -2}_\beta
\Big|\sum_{j=1 }^m e_j X_{ i }^{\zeta}e_j +  \beta^2 \sum_{j=m+1 }^n e_j X_{i}^{\zeta}e_j\Big|
\geq \mu\epsilon^{1-\frac{\gamma}{2}} d^{\gamma -2}_\beta\Big(
\frac{|e_i|}{2} -  \beta d_\beta d_0\Big),
$$
and therefore
\begin{equation}	\label{stimagrad}
|\nabla_0 ^{\zeta} \phi| \geq \mu\epsilon^{1-\frac{\gamma}{2}} d^{\gamma -2}_\beta d_0. 
\end{equation}

\bigskip

If $\operatorname{deg}i=2$ and $\operatorname{deg}j=1,$ 
we obtain 
\begin{equation}\label{stima21}
X_{\delta i}^{\xi}e_j =O(\delta d_\beta).
\end{equation}
Indeed in the case of Carnot groups we have 
$$X_{\delta i}^{\xi}e_j=X_{\delta i}^{\eta}e_j=0,$$
while in $\se$ we obtain
$$X^\xi_{3\delta} e_1 =X^\eta_{3\delta} e_1=-\delta\sin(\frac{\hat{\theta}_{\xi}-\hat{\theta}_{\eta}}{2} )=-\frac{\delta e_2}{2}+O(\delta d_{\beta}^3)  ,\quad \quad X^\xi_{3\delta} e_2 =X^\eta_{3\delta} e_2 =  0. $$
Finally if $\operatorname{deg}i=2=\operatorname{deg}j=2,$
\begin{equation}\label{stima22}
X_{\delta i}^{\xi}e_j  =\delta \delta_{ij} + O(\delta d^2_\beta), \quad  X_{\delta i}^{\eta}e_j  =-\delta \delta_{ij}+O(\delta d^2_\beta), \quad {\color{green}|} X_{\delta i}^{\xi}e_j + X_{\delta i}^{\eta}e_j{\color{green}|}= 0.
\end{equation}
In Carnot groups
\begin{equation}\label{eq:Carnot_case_0_deg2}
X^\xi_{i\delta}e_j=-X^\eta_{i\delta}e_j=\delta\,\delta_{ij}.
\end{equation}
In $\se$, we obtain
\begin{equation}\label{eq:se2_0_deg2}
X^\xi_{3\delta} e_3 = \delta \cos(\frac{\hat{\theta}_{\xi}-\hat{\theta}_{\eta}}{2} )=  \delta + O(\delta e^2_2), \quad 
X^\eta_{3\delta} e_3=-\delta \cos(\frac{\hat{\theta}_{\xi}-\hat{\theta}_{\eta}}{2})=-\delta+O(\delta e^2_2).
\end{equation}

Now we plug \eqref{stima21} and \eqref{stima22} 
in  \eqref{derivaphi} and we obtain for $\operatorname{deg} i = 2$, and $\zeta=\xi$ or $\zeta=\eta$
\begin{equation}\label{eq:deg2_firstOrederStim}
|X_{ i\delta}^{\zeta} \phi| \leq 
\mu\epsilon^{1-\frac{\gamma}{2}}  d^{\gamma -2}_\beta
\Big|\sum_{j=1 }^m e_j X_{ i\delta}^{\zeta}e_j +  \beta^2 \sum_{j=m+1 }^n e_j X_{i\delta}^{\zeta}e_j\Big|\leq
\tilde{K}\mu\epsilon^{1-\frac{\gamma}{2}} d^{\gamma-2}_\beta d_0
\delta.
\end{equation}
Always from \eqref{derivaphi} we have 
$$|X_{ i\delta}^{\xi} \phi + X_{ i\delta}^{\eta} \phi | \leq 
\mu\epsilon^{1-\frac{\gamma}{2}} d^{\gamma -2}_\beta
\Big|  \sum_{j=1 }^m e_j X_{ i\delta}^{\zeta}e_j \Big|\leq
\tilde{K}\mu\epsilon^{1-\frac{\gamma}{2}} d^{\gamma}_\beta
\delta.
$$
Then the inequalities in \eqref{remarkderivative} are directly found.

Finally applying \eqref{derivaphi} with $\deg i=2$ and the fact that 
$|X_{i\delta}^{\zeta}e_j|\geq \delta_{ij}/2,$ if $\deg i=\deg j=2$ we have
$$|X_{ i\delta}^{\zeta} \phi| \geq 
\mu\epsilon^{1-\frac{\gamma}{2}}  d^{\gamma -2}_\beta
\Big|\sum_{j=1 }^m e_j X_{ i\delta}^{\zeta}e_j +  \beta^2 \sum_{j=m+1 }^n e_j X_{i\delta}^{\zeta}e_j\Big|\geq
\mu\epsilon^{1-\frac{\gamma}{2}}  d^{\gamma -2}_\beta\Big(
 \beta^2 \delta | e_i |/2- \Big|\sum_{j=1 }^m e_j X_{ i\delta}^{\zeta}e_j  \Big|	\Big)$$
$$\geq 
\mu\epsilon^{1-\frac{\gamma}{2}}  d^{\gamma -2}_\beta\Big(
 \beta^2 \delta | e_i |/2- \delta d_\beta d_0 \Big).
$$
Also using \eqref{stimagrad} we obtain
$$
|\nabla_\delta ^{\zeta} \phi| \geq \mu\epsilon^{1-\frac{\gamma}{2}} d^{\gamma -2}_\beta (d_0 + 
\beta \delta d_3).$$

%
\end{proof}

\begin{lemma}\label{lem:zero_gradient_implies}
At the point $(\hxi,\het,\htt)$, the following holds:
$$\abs{\nabla_0^{\xi}\phi}^2=0 \implies \abs{\nabla_0^{\eta}\phi}^2=0.$$
\end{lemma}
\begin{proof}
In Carnot groups, the zero gradient at $(\hxi,\het,\htt)$ implies
\begin{equation}
d_0^{2}=0, 
\end{equation}
therefore
\begin{equation}
X_{i}^{\xi}\phi=0,\quad X_{i}^{\eta}\phi=0,\quad\text{if $\operatorname{deg}i=1$}. 
\end{equation}
Furthermore \eqref{eq:carnot_gradient_equal} gives 
\begin{equation}
\abs{\nabla_0^{\eta}\phi}^2=0.
\end{equation}
In $\se$, we first notice that $X_1^{\xi}\phi(\hat{\xi},\hat{\eta},\hat{t})=0$ and $X_2^{\xi}\phi(\hat{\xi},\hat{\eta},\hat{t})=0$ imply
\begin{equation}\label{eq:zero_gradient_observation}
e_1\cos(\frac{\hat{\theta}_{\xi}-\hat{\theta}_{\eta}}{2})+\beta^2 e_3\sin(\frac{\hat{\theta}_{\xi}-\hat{\theta}_{\eta}}{2} )=0,\quad \hat{\theta}_\xi-\hat{\theta}_\eta=0,
\end{equation}
respectively, as a result of $\hat{\theta}_{\xi}-\hat{\theta}_{\eta}\ll 1$. From \eqref{eq:zero_gradient_observation}, we see also that
$$e_1=0,$$
therefore
$$\nabla_0^{\eta}\phi=0.$$
\end{proof}

Let us now collect some properties of the second derivatives of the function $\phi$. 
\begin{lemma}\label{rem:SecondOrderEst_Lemma_both}
Assume that $d_{\beta}\leq 1$. At the point $(\hat{\xi},\hat{\eta},\hat{t})$,
if $\operatorname{deg}i=\operatorname{deg}j=1$, 
we have 
\begin{equation}\label{rem:secondOrderEst1}
|X^\eta_{i}X^\eta_{j} \phi | \leq  \tilde{K}\mu\eps^{1-\frac{\gamma}2}d_{\beta}^{\gamma-2}.
\end{equation}

If  $\deg i=2$  or $\deg j=2$, we have
\begin{equation}\label{derivateij12}
X^\eta_{i\delta}X^\eta_{j\delta} \phi \leq  \tilde{K}\mu\eps^{1-\frac{\gamma}2}\delta d_{\beta}^{\gamma-2}.\end{equation}

If $\operatorname{deg}i=\operatorname{deg}j=2$
\begin{equation}\label{derivateij22}
X^\eta_{i\delta}X^\eta_{j\delta} \phi  \leq  \tilde{K}\mu\eps^{1-\frac{\gamma}{2}}\delta^2 d_{\beta}^{\gamma-2}. 
\end{equation}


Finally, in both settings for $i$ and $j$ of any degree the following holds:
\begin{align}\label{rem:for_the_Hessian_estimates}
\begin{split}
|X^\xi_{i}(X^\xi_{j} \phi+X^\eta_{j}\phi)|, \quad  |X^\eta_{i}(X^\xi_{j}\phi +X^\eta_{j}\phi)| & \leq \tilde{K}\mu\eps^{1-\frac{\gamma}{2}}\beta d_{\beta}^{\gamma-2}.
\end{split}
\end{align}
\end{lemma}

\begin{proof}

It is a direct computation using the expression of the derivatives 
already computed in the previous lemma together with the second order derivatives of the increments. 
Note that for $\zeta_1,\zeta_2\in \{\xi, \eta \}$ a second order derivative of $\phi$ is written as
\begin{align}\label{eq:secondOrder_derivative_expression}
\begin{split}
X^{\zeta_1}_{i\delta} X^{\zeta_2}_{j\delta}\phi= & (\gamma-2)\gamma\mu\eps^{1-\frac{\gamma}{2}}d_{\beta}^{\gamma-4}\Big( \sum_{k=1 }^m e_k X_{i\delta}^{\zeta_1}e_k+\beta^2\sum_{k=m+1 }^n e_k X_{i\delta}^{\zeta_1}e_k \Big)\Big( \sum_{k=1 }^m e_k X_{i\delta}^{\zeta_2}e_k+\beta^2\sum_{k=m+1 }^n e_k X_{i\delta}^{\zeta_2}e_k \Big)\\
 &+\gamma\mu\eps^{1-\frac{\gamma}{2}}d_{\beta}^{\gamma-2}\sum_{k=1 }^m\Big(X_{i\delta}^{\zeta_1}e_k X_{j\delta}^{\zeta_2}e_k +  e_k X_{i\delta}^{\zeta_1} X_{j\delta}^{\zeta_2}e_k\Big)\\
 & +\gamma\mu\eps^{1-\frac{\gamma}{2}}\beta^2 d_{\beta}^{\gamma-2} \sum_{k=m+1}^n \Big(X_{i\delta}^{\zeta_1}e_k X_{j\delta}^{\zeta_2}e_k + e_k X_{i\delta}^{\zeta_1}X_{j\delta}^{\zeta_2}e_k \Big).
\end{split}
\end{align}

We plug the first and second order derivatives of the increments in \eqref{eq:secondOrder_derivative_expression}. Then for $\deg i=\deg j=1$ and $\zeta$ denoting either $\xi$ or $\eta$,  we see
\begin{equation}
\label{eq:est_sec_dist}
\sum_{k=1 }^m\Big(X_{i}^{\zeta}e_k X_{ j}^{\zeta}e_k +  e_k X_{i}^{\zeta} X_{ j}^{\zeta}e_k\Big)\leq 
\tilde{K}(1 + d^2_\beta)\leq \tilde{K}, \quad  \beta^2 \sum_{k=m+1 }^n \Big(X_{i}^{\zeta}e_k X_{ j}^{\zeta}e_k + e_k X_{i}^{\zeta}X_{ j}^{\zeta}e_k \Big)\leq \beta d_\beta,
\end{equation}
and we obtain by using also Lemma \ref{rem:fistOrderEst_both} that
$$X_i^{\zeta}X_j^{\zeta}\phi\leq \tilde{K}\mu  
\eps^{1-\frac{\gamma}2} d_\beta^{\gamma-4}(d_\beta+\beta d_{\beta}^{2})^2  + \tilde{K}\mu  
\eps^{1-\frac{\gamma}2} d_\beta^{\gamma-2}(1 + \beta d_\beta)\leq \tilde{K}\mu\eps^{1-\frac{\gamma}{2}}d_{\beta}^{\gamma-2},
$$
which gives the estimate in \eqref{rem:secondOrderEst1}.


The estimates given by \eqref{derivateij12} and \eqref{derivateij22} are found in the same way via direct substitution of the first and second order derivatives of the increments in \eqref{eq:secondOrder_derivative_expression}.

Finally the estimate provided in \eqref{rem:for_the_Hessian_estimates} is found via differentiating \eqref{sommaderivaphi} straightforwardly. 
\end{proof}

Let us now provide the proof of our main result, Theorem \ref{thm:univiscosityTheorem}. 

\begin{proof}
We argue by contradiction and assume that for  each $M\geq 0$ there exists an $\eps$ such that $0<\delta<\eps<1$ and
\begin{equation}
\sup_{\xi\in G, 0< t\leq T}\lvert (u-u^{\eps}_{\delta})(\xi,t)\rvert >M\eps^{\alpha}.
\end{equation}
Due to Theorem \ref{existence} and Theorem \ref{thm:exponentialEstimateThmHeisenberg} we can apply Lemma \ref{lemma:lem41} and deduce that the function $\omega$ 
defined in \eqref{eq:defnOmega2} attains its maximum at an interior point 
$(\hxi,\het, \htt)$ satisfying 
\begin{equation}\label{eq:distEstimateAbsol}
d_\beta(\hat \xi, \hat \eta)\leq \Big( \frac{2\gamma \tilde{C}}{\mu}\eps^{\frac{\gamma}{2}-1}-\frac{M\gamma}{4\mu}\eps^{\alpha+\frac{\gamma}{2}-1} \Big)^{\frac{1}{\gamma}}.
\end{equation}

Since the function $u$ is only continuous, we replace the derivatives of $u$ with the superjet $\mathcal{P}^{2,+}u(\xi,t)$ (see Subsection \ref{sec:superjetDefns} for the superjet definition). Note that $u^{\eps}_{\delta}$ is smooth so we have $-X_i^{\eta}u^{\eps}_{\delta}(\het,\htt)= X_i^{\eta}\phi(\hxi,\het,\htt)$  
and $-X_i^{\eta}X_j^{\eta}u^{\eps}_{\delta}(\het,\htt)\leq X_i^{\eta}X_j^{\eta}\phi(\hxi,\het,\htt)$. 
We know by \cite[Theorem 8.3]{crandall1992user} that for every $\rho>0$ there exist symmetric matrices $H= (H_{ij})_{i,j, =1, \cdots, m}$ and $Y = (Y_{ij})_{i,j, =1, \cdots, m}$ such that
\bit
\item[(i)] $\quad\quad\quad (a, \nabla_0^{\xi}\phi(\hxi,\het,\htt),H)\in\jetp u(\hxi,\htt),$

$\quad\quad\quad(b, \nabla_{\delta}^{\eta}\phi(\hxi,\het,\htt),Y)\in -\jetn u^{\eps}_{\delta}(\het,\htt)$,

\item[(ii)] $\quad\quad\quad  a+b=\frac{M}{2T}\eps^{\alpha}$,
\item[(iii)]  $\quad\quad\quad -\Big(\frac{1}{\rho}+ \vert\vert A \vert\vert\Big) I\leq \begin{pmatrix}
H & 0 \\ 0 & Y 
\end{pmatrix}\leq A+\rho A^2$,
\eit   
where $A$ is defined as follows:
\begin{equation}
A=\begin{pmatrix}
B_{\xi\xi} & B_{_{\xi\eta}} \\ B_{\eta\xi} & B_{\eta\eta}
\end{pmatrix},
 \quad B_{\zeta_1\zeta_2}=
\begin{pmatrix}
X_i^{\zeta_1}X_i^{\zeta_2}\phi & \frac{1}{2}( X_i^{\zeta_1}X_j^{\zeta_1}\phi+X_j^{\zeta_1}X_i^{\zeta_2}\phi)\\
\frac{1}{2}( X_i^{\zeta_1}X_j^{\zeta_2}\phi+X_j^{\zeta_1}X_i^{\zeta_2}\phi) & X_j^{\zeta_1}X_j^{\zeta_2}\phi
\end{pmatrix}_{\text{$i,j=1,\dots,m$}\atop \text{$i\neq j$}}
\end{equation}
where $\zeta_1,\zeta_2\in\{\xi,\eta \}$.

We employ \eqref{rem:for_the_Hessian_estimates} and (iii) to see that for all $z\in \R^m$ 
\begin{align}\label{eq:condition_3_form_CIL}
z^T(H+Y)z=
\begin{pmatrix}
z^T, & z^T
\end{pmatrix}
\begin{pmatrix}
H & 0\\
0 & Y
\end{pmatrix}
\begin{pmatrix}
z \\ z
\end{pmatrix}
\leq\begin{pmatrix}
z^T, & z^T
\end{pmatrix}
(A+\rho A^2)
\begin{pmatrix}
z \\ z
\end{pmatrix}.
\end{align}

Due to \eqref{rem:for_the_Hessian_estimates}, this implies that for $0<\beta< 1$ and as $\rho$ goes to zero we have
\begin{equation}\label{eq:HY_estimate}
z^T(H+Y)z\leq \tilde{K}\mu\eps^{1-\frac{\gamma}{2}}\beta d_{\beta}^{\gamma-2}\abs{z}^2.
\end{equation}

Since $(a,\nabla_0^{\xi}\phi(\hxi,\het,\htt),H)\in\jetp 
u(\hxi,\htt)$ and $u$ is a viscosity subsolution of \eqref{eq:orgMCF} we have
\begin{align}\label{eq:viscDefinitionInProof}
\begin{split}
a & -\SUM_{i,j=1}^m\Big(\delta_{ij}-\frac{X_i^{\xi}\phi(\hxi,\het,\htt) X_j^{\xi}\phi(\hxi,\het,\htt)}{\abs{\nabla_0^{\xi}\phi(\hxi,\het,\htt)}^2}   \Big){\color{black}H}_{ij}\leq 0,\quad \text{if}\quad \nabla_0^{\xi}{\phi(\hxi,\het,
\htt)}\neq 0,\\
a & -\SUM_{i,j=1}^m\delta_{ij}{\color{black}H}_{ij}\leq 0,\quad\text{if}\quad \nabla_0^{\xi}{\phi(\hxi,\het,
\htt)}= 0.
\end{split}
\end{align}

Moreover since $u^{\eps}_{\delta}$ is smooth we have $\partial_t u^{\eps}_{\delta}(\het,\htt)=-b$. It is a solution and it satisfies
\begin{align}\label{eq:uesolutionDefinitionInProof}
\begin{split}
b  & = - \partial_t u^{\eps}_{\delta}(\het,\htt)=
-\SUM_{i,j=1}^n\Big(\delta_{ij}-\frac{X_{i\delta}^{\eta} u^{\eps}_{\delta}(\hxi,\het,\htt) X_{j\delta}^{\eta} u^{\eps}_{\delta}(\hxi,\het,\htt)}{\eps^2+\abs{\nabla_\delta^{\eta} u^{\eps}_{\delta}(\hxi,\het,\htt)}^2}   \Big)X^{\eta}_{i\delta}X^{\eta}_{j\delta} u^{\eps}_{\delta}(\het,\htt)
\\
&\leq \SUM_{i,j=1}^n\Big(\delta_{ij}-\frac{X_{i\delta}^{\xi} \phi(\hxi,\het,\htt) 
X_{j\delta}^{\xi} \phi(\hxi,\het,\htt)}{\eps^2+\abs{\nabla_\delta^{\xi} \phi(\hxi,\het,\htt)}^2}\Big)X^{\eta}_{i\delta}X^{\eta}_{j\delta} \phi(\het,\htt),
\end{split}
\end{align}
due to the fact that at the maximum point 
$-\nabla_\delta^{\eta} u^{\eps}_{\delta}(\het,\htt)=\nabla_\delta^{\eta}\phi(\hxi,\het,\htt),$
and  
\begin{equation}\label{inequalities}
- X^{\eta}_iX^{\eta}_j u_{\delta}^\eps(\het,\htt) \leq X^{\eta}_iX^{\eta}_j \phi(\het,\htt), \quad 
- X^{\eta}_iX^{\eta}_j u_{\delta}^\eps(\het,\htt)\leq Y_{ij}.
\end{equation}

Now we distinguish two cases: 
$\nabla_0^{\xi}\phi (\hxi,\het,\htt)\neq 0$ and 
$\nabla_0^{\xi}\phi (\hxi,\het,\htt)=0$. Let us first consider the case with $\nabla_0^{\xi}\phi (\hxi,\het,\htt)\neq 0$. 

We sum \eqref{eq:viscDefinitionInProof} and \eqref{eq:uesolutionDefinitionInProof}. Then we obtain
\begin{align}\label{eq:compactForm11}
\begin{split}
&\frac{M \eps^{\alpha}}{2T}= a+b
\\& \leq\SUM_{i,j=1}^m\Big( \delta_{ij}-\frac{X_i^{\xi}\phi X_j^{\xi}\phi}{\abs{\nabla_0^{\xi}\phi}^2} \Big)H_{ij}+
\SUM_{i,j=1}^n\Big( \delta_{ij}-
\frac{X_{i\delta}^{\eta}\phi X_{j\delta}^{\eta}\phi}{(\eps^2+\abs{\nabla_\delta^{\eta} \phi}^2) }\Big)
X^{\eta}_{i\delta}X^{\eta}_{j\delta}\phi
\\& \leq\SUM_{i,j=1}^m\Big( \delta_{ij}-\frac{X_i^{\xi}\phi X_j^{\xi}\phi}{\abs{\nabla_0^{\xi}\phi}^2} \Big)(H_{ij}+Y_{ij})+ 2\SUM_{i=1}^n\SUM_{j=m+1}^n\Big | \delta_{ij}-
\frac{X_{i\delta}^{\eta}\phi X_{j\delta}^{\eta}\phi}{(\eps^2+\abs{\nabla_\delta^{\eta} \phi}^2) }\Big |
X^{\eta}_{i\delta}X^{\eta}_{j\delta}\phi
\\& + \SUM_{i,j=1}^mX_i^{\xi}\phi X_j^{\xi}\phi\Big(\frac{ \eps^2 }{(\eps^2+\abs{\nabla_{\delta}^{\eta} \phi}^2)\abs{\nabla_0^{\xi}\phi}^2}\Big)X_i^{\eta}X_j^{\eta}\phi
\\ &
+\SUM_{i,j=1}^m X_i^{\xi}\phi X_j^{\xi}\phi\Big(\frac{\abs{\nabla_{\delta}^{\eta}\phi}^2-\abs{\nabla_0^{\eta}\phi}^2\, }{(\eps^2+\abs{\nabla_\delta^{\eta} \phi}^2)\abs{\nabla_0^{\xi}\phi}^2 }\Big)X^{\eta}_i X_j^{\eta}\phi
\\ &
+\SUM_{i,j=1}^m\Big(\frac{ X_i^{\xi}\phi X_j^{\xi}\phi (\abs{\nabla_0^{\eta}\phi}^2-\abs{\nabla_0^{\xi}\phi}^2)+\abs{\nabla_0^{\xi}\phi}^2 (X_i^{\xi}\phi X_j^{\xi}\phi-X_i^{\eta}\phi X_j^{\eta}\phi)  }{(\eps^2+\abs{\nabla_\delta^{\eta} \phi}^2)\abs{\nabla_0^{\xi}\phi}^2 }\Big) X^{\eta}_i X_j^{\eta}\phi .
\end{split}
\end{align}
Let us denote by $I_1,\,I_2,\,I_3,\,I_4,\,I_5$ the five terms appearing on the right hand side of the last inequality in \eqref{eq:compactForm11}, and let us 
consider one term at a time. 

We set $\beta=\delta=\eps^{\sigma}$, where $\sigma>0$ is to be determined in the sequel. Furthermore, we will employ the estimate which is found from \eqref{eq:distEstimateAbsol}: 
\begin{equation}\label{eq:estimateSimple_for_distance}
d_\beta^{\gamma}\leq \frac{2\tilde{C}\gamma}{\mu}\epsilon^{\frac{\gamma}{2}-1}.
\end{equation}
Finally we note that for any value of $M$ we can find an $\eps$ (which satisfies $0<\eps<1$) such that 
$$d_{\beta}(\hat{\xi},\hat{\eta})\leq 1.$$

We use \eqref{eq:HY_estimate} and observe that
\begin{align}\label{eq:HY_I1}
\begin{split}
I_1=
\SUM_{i,j=1}^m\Big( \delta_{ij}-\frac{X_i^{\xi}\phi X_j^{\xi}\phi}{\abs{\nabla_0^{\xi}\phi}^2} \Big) (H_{ij}+Y_{ij})\leq \tilde{K}\mu\eps^{1-\frac{\gamma}{2}}\beta d_{\beta}^{\gamma-2}.
\end{split}
\end{align}

We use \eqref{eq:estimateSimple_for_distance} and the definition of $\mu$ given by \eqref{eq:muExpression} in this expression. Then we find
\begin{align}\label{eq:HY_I1_2}
\begin{split}
\tilde{K}\mu\eps^{1-\frac{\gamma}{2}}\beta d_{\beta}^{\gamma-2}&\leq \tilde{K}(\mu^{\frac{1}{\gamma}}\eps^{(1-\frac{\gamma}{2})\frac{1}{\gamma}+\sigma}+\mu^{\frac{2}{\gamma}}\eps^{(1-\frac{\gamma}{2})\frac{2}{\gamma}+2\sigma})  \leq
\tilde{K}\mu^{\frac{2}{\gamma}}\eps^{(1-\frac{\gamma}{2})\frac{2}{\gamma}+\sigma}
,
\end{split}
\end{align}
which implies for $M>1$ that
\begin{equation}
I_1\leq\tilde{K}\mu^{\frac{2}{\gamma}}\eps^{(1-\frac{\gamma}{2})\frac{2}{\gamma}+\sigma}.
\end{equation}
Once we impose on $\sigma>0$, the condition given by
\begin{equation}\label{eq:sigma_condition_carnot_1}
(1-\frac{\gamma}{2})\frac{2}{\gamma}+\sigma \geq \alpha,
\end{equation}
results for $M>1$ in
\begin{equation}
I_1\leq \tilde{K}M^{-1}\eps^{\alpha}.
\end{equation}

We continue with the second term: $I_2$. We first notice that
\begin{equation}\label{equazioneI2}
I_2 \leq 2\SUM_{i=1}^{n} 
\SUM_{j=m+1}^{n}|X_{i\delta}^{\eta}X_{j\delta}^{\eta}\phi|  .
\end{equation}
Furthermore, by \eqref{derivateij12} and \eqref{derivateij22}, we have
\begin{align}\label{eq:threeInequalities}
|X_{i\delta}^{\eta}X_{j\delta}^{\eta}\phi| \leq\tilde{K}\mu\eps^{1-\frac{\gamma}{2}}(\delta^2 d_{\beta}^{\gamma-2}+ \beta^2\delta^2 d_\beta^{\gamma-4})
\tilde{K}\mu\eps^{1-\frac{\gamma}{2}}\delta^2 d_{\beta}^{\gamma-2}\leq \tilde{K}\mu\eps^{1-\frac{\gamma}{2}}\delta d_{\beta}^{\gamma-2}
\end{align}
if at least one of $i$ and $j$ is of degree 2. 
We employ \eqref{eq:muExpression} and \eqref{eq:estimateSimple_for_distance} to find 
\begin{equation}
I_2\leq \mu^{\frac{2}{\gamma}}\eps^{(1-\frac{\gamma}{2})\frac{2}{\gamma}+\sigma}
\end{equation}
which implies for $M>1$ and the values of $\sigma$ satisfying \eqref{eq:sigma_condition_carnot_1} that
\begin{equation}
I_2\leq \tilde{K}M^{-1}\eps^{\alpha}.
\end{equation}

Now we consider $I_3$. First we observe that
\begin{equation}
 I_3\leq \SUM_{i,j=1}^m\frac{\eps^2\,\abs{X_i^{\eta}X_j^{\eta}\phi}}{(\eps^2+\abs{\nabla_{\delta}^{\eta} \phi}^2)}.
 \end{equation} 
Then we employ the estimates given in \eqref{remarknabla} and  \eqref{rem:secondOrderEst1} in order to find
\begin{align}\label{eq:YoungInq_based_starts}
\begin{split}
\SUM_{i,j=1}^m\frac{\eps^2\,X_i^{\eta}X_j^{\eta}\phi}{(\eps^2+\abs{\nabla_{\delta}^{\eta} \phi}^2)}\leq & \frac{\tilde{K}(\gamma-2)\gamma\mu\eps^{3-\frac{\gamma}{2}} d_{\beta}^{\gamma-2}}{\eps^2+\gamma^2\mu^2\eps^{2-\gamma} d_\beta^{2\gamma-4}
(d_0^2 + \beta^2 \delta^2 d_3^2 )}\leq 
\frac{\tilde{K}(\gamma-2)\gamma\mu\eps^{3-\frac{\gamma}{2}} d_0^{\gamma-2} + 
(\gamma-2)\gamma\mu\eps^{3-\frac{\gamma}{2}} d_3^{\gamma-2}
}{\eps^2+\gamma^2\mu^2\eps^{2-\gamma} d_0^{2\gamma-2} + \gamma^2\mu^2\eps^{2-\gamma} \beta^2 \delta^2 d_3^{2\gamma-2}}.
\end{split}
\end{align}
We may write the first term in the numerator as follows: 
\begin{equation}
(\gamma-2)\gamma\mu\eps^{3-\frac{\gamma}{2}}\gamma d_0^{\gamma-2} =
(\gamma-2)\gamma\mu^{\frac{1}{\gamma-1}}\eps^{\frac{\gamma-2}{2(\gamma-1)}}
\eps^{\frac{\gamma}{\gamma-1}}\mu^{\frac{\gamma-2}{\gamma-1}}\eps^{-\frac{(\gamma-2)^2}{2(\gamma-1)}}d_0^{\gamma-2}.
\end{equation}
Then we employ Young's inequality $mn\leq \frac{1}{p}m^p+\frac{1}{q}n^q$  with $p=\frac{2(\gamma-1)}{\gamma}$ and $q=\frac{2(\gamma-1)}{\gamma-2}$, which results in
$$
(\gamma-2)\gamma\mu\eps^{3-\frac{\gamma}{2}}\gamma d_0^{\gamma-2} \leq 
(\gamma-2)\gamma\mu^{\frac{1}{\gamma-1}}\eps^{\frac{\gamma-2}{2(\gamma-1)}}
\Big(\frac{\gamma}{2(\gamma-1)}\eps^2+\frac{\gamma-2}{2(\gamma-1)}\mu^2\eps^{2-\gamma}d_0^{2\gamma-2}
\Big) $$
\begin{equation}\label{pezzo1}
\leq \tilde{K}(\gamma-2)\gamma\mu^{\frac{1}{\gamma-1}}\eps^{\frac{\gamma-2}{2(\gamma-1)}}
\Big(\frac{1}{2}\eps^2+\mu^2\eps^{2-\gamma}d_0^{2\gamma-2}
\Big).
\end{equation}
The second term in the numerator can be handled in the same way:
$$(\gamma-2)\gamma\mu\eps^{3-\frac{\gamma}{2}}\gamma d_3^{\gamma-2}
=(\gamma-2)\gamma\mu^{\frac{1}{\gamma-1}}\eps^{\frac{(\gamma-2)(1- 4\sigma )}{2(\gamma-1)}  }
\eps^{\frac{\gamma}{\gamma-1}}\mu^{\frac{\gamma-2}{\gamma-1}}\eps^{-\frac{(\gamma-2)(\gamma-2 - 4\sigma)}{2(\gamma-1)}}d_3^{\gamma-2}
$$
\begin{equation}\label{pezzo2}
\leq \tilde{K}(\gamma-2)\gamma\mu^{\frac{1}{\gamma-1}}
\eps^{\frac{(\gamma-2)(1- 4\sigma )}{2(\gamma-1)}  }
\Big(\frac{1}{2}\eps^2+\mu^2\beta^2\delta^2\eps^{2-\gamma}d_3^{2\gamma-2}
\Big).
\end{equation}

We sum up \eqref{pezzo1} and \eqref{pezzo2} and obtain
\begin{align}
\SUM_{i,j=1}^m\frac{\eps^2\,X_i^{\eta}X_j^{\eta}\phi}{(\eps^2+\abs{\nabla_{\delta}^{\eta} \phi}^2)} \leq \tilde{K}\mu^{\frac{1}{\gamma-1}}
\eps^{\frac{(\gamma-2)(1- 4\sigma )}{2(\gamma-1)}   }.
\end{align}
Furthermore we use \eqref{eq:muExpression} and find
\begin{align}
\tilde{K}\mu^{\frac{1}{\gamma-1}}
\eps^{\frac{(\gamma-2)(1- 4\sigma )}{2(\gamma-1)}   } \leq & \tilde{K} \Big(\frac{\gamma 4^{\gamma}\operatorname{Lip}(u_0)^{\gamma}}{M^{\gamma-1}}\Big)^{\frac{1}{\gamma-1}}\eps^{\frac{(\gamma-2)(1- 4\sigma )}{2(\gamma-1)}},
\end{align}
which implies for $M>1$ and the values of $\sigma$ satisfying \eqref{eq:sigma_condition_carnot_1} the following:
\begin{equation}\label{eq:YoungInq_based_ends}
I_3\leq \tilde{K} \eps^{\frac{(\gamma-2)(1- 4\sigma )}{2(\gamma-1)}  }M^{-1} = \tilde K M^{-1} \epsilon^\alpha 
\end{equation}
if 
\begin{equation}\label{eq:sigma_condition_carnot_3}
\alpha = \frac{(\gamma-2)(1- 4\sigma )}{2(\gamma-1)}.
\end{equation}

Now let us tackle $I_4$ and $I_5$. Note that
by \eqref{firstderivphi}
\begin{equation}\label{eq:carnot_gradient_equal}
\Big|\abs{\nabla^{\eta}_0\phi}^2- \abs{\nabla^{\eta}_0\phi}^2\Big|\leq 
\sum_{i=m+1}^n |X_i^\eta\phi|^2\leq 
\delta^2	\tilde{K}\mu^2  \eps^{2-\gamma} d_\beta^{2\gamma-4}d^2_0\leq \tilde K\delta^2 \abs{\nabla^{\eta}_0\phi}^2,
\end{equation}
hence
\begin{equation}\label{eq:startL_3}
I_4\leq \tilde K \delta^2\SUM_{i,j=1}^m X^{\eta}_iX^{\eta}_j\phi . 
\end{equation}

Now, using \eqref{firstderivphi} and \eqref{remarknabla}
we obtain
$$\Big| X_i^{\xi}\phi X_j^{\xi}\phi (\abs{\nabla_0^{\eta}\phi}^2-\abs{\nabla_0^{\xi}\phi}^2)+\abs{\nabla_0^{\xi}\phi}^2 (X_i^{\xi}\phi X_j^{\xi}\phi-X_i^{\eta}\phi X_j^{\eta}\phi)\Big|\leq 
 $$
$$\abs{\nabla_0^{\xi}\phi}^2\sum_{i,j=1}^m(X_j^{\xi}\phi+ X_j^{\eta}\phi) (X_i^{\xi}\phi -X_i^{\eta}\phi)\leq \abs{\nabla_0^{\xi}\phi}^2 
\beta \tilde{K}^2\mu^2  \eps^{2-\gamma} d_\beta^{2\gamma-4}d^2_0\leq 
\beta  \abs{\nabla_0^{\xi}\phi}^2  \abs{\nabla_0^{\eta}\phi}^2, 
$$
As a consequence 
$$I_5 \leq \tilde K \beta \SUM_{i,j=1}^m X^{\eta}_iX^{\eta}_j\phi.$$

At this point we use the estimates given in \eqref{remarkderivative}, \eqref{rem:secondOrderEst1} and find for $M>1$ that
\begin{align}
\begin{split}
I_4 + I_5\leq & (\delta + \beta)\tilde{K}\mu\eps^{1-\frac{\gamma}{2}}d_{\beta}^{\gamma-2}.
\end{split}
\end{align}
This term is clearly similar to $I_2$. We see that for the values of $\sigma$ which satisfy \eqref{eq:sigma_condition_carnot_1} we have
\begin{equation}
I_4 + I_5\leq \tilde {K}M^{-1}\eps^{\alpha}.
\end{equation}

Finally we have proved  that 
\begin{equation}\label{eq:contradiction_nonzero}
\frac{M \epsilon^\alpha}{2T} \leq I_1+I_2+I_3+I_4+I_5\leq \tilde{K}M^{-1} \epsilon^\alpha,
\end{equation}
where $\sigma$ fulfills the condition given in \eqref{eq:sigma_condition_carnot_1} and $\alpha>0$ is defined as in \eqref{eq:sigma_condition_carnot_3}. This result is a contradiction for large values of $M$, proving the assertion of the theorem if the gradient of $\phi$ does not vanish at the maximum point $(\hat{\xi}, \hat{\eta},\hat{t})$.

Now we look at the case with $\nabla_0^{\xi}\phi(\hxi,\het,\htt)=0$. In this case we can write \eqref{eq:compactForm11} by using \eqref{eq:viscDefinitionInProof} and Lemma \ref{lem:zero_gradient_implies} as
\begin{align}\label{eq:contr_disuguglianza}
\begin{split}
\frac{M}{2T}\eps^{\alpha}=  a+b 
\leq  \SUM_{i,j=1}^m \delta_{ij}({\color{black}H}_{ij}+Y_{ij})+2\SUM_{i=1}^n\SUM_{j=m+1}^n\Big( \delta_{ij}-
\frac{X_{i\delta}^{\eta}\phi X_{j\delta}^{\eta}\phi}{(\eps^2+\abs{X_{3\delta}^{\eta} \phi}^2) }\Big)
X^{\eta}_{i\delta}X^{\eta}_{j\delta}\phi.
\end{split}
\end{align}

The first and second terms on the right hand side of the final inequality can be tackled exactly in the same way employed in the previous non-zero gradient case of $I_1$ and $I_2$, respectively. Then the contradiction follows from \eqref{eq:contradiction_nonzero} as before.

We have obtained the contradiction for both cases with vanishing and non-vanishing horizontal gradients. The proof is complete. 
\end{proof}

\begin{remark}
The comparison principle given in Theorem \ref{thm:comparisonForVanishingViscosity} was valid only for the vanishing viscosity solutions. Now we have also the comparison principle for the viscosity solutions both in Carnot groups of step 2 and $\se$, as given in Corollary \ref{thm:comparisonPrincipleViscosity}. It follows directly from Theorem \ref{thm:univiscosityTheorem}.
\end{remark}

\bibliographystyle{siam}
\bibliography{uniquenessPreprintBib}

\end{document}